\pdfoutput=1 

\documentclass[a4paper,twoside]{article}

\usepackage[british]{babel}

\usepackage{amsmath}       %
\usepackage{amssymb}       %
\usepackage{amsthm}        %
\usepackage{mathscinet}    %
\usepackage{mathrsfs}      
\usepackage{hyperref}      
\usepackage{microtype}     
\usepackage{tabularx}      
\usepackage[T1]{fontenc}   
\usepackage{lmodern}       
\usepackage[utf8]{inputenc}
\usepackage{enumitem}      

\swapnumbers
\newtheorem{theorem}{Theorem}[section]           
\newtheorem{corollary}[theorem]{Corollary}       
\newtheorem{lemma}[theorem]{Lemma}               
\theoremstyle{definition}   
\newtheorem{definition}[theorem]{Definition}     
\newtheorem{miniremark}[theorem]{}               
\newtheorem{example}[theorem]{Example}           
\theoremstyle{remark}
\newtheorem{remark}[theorem]{Remark}             

\newtheoremstyle{citing}                
  {3pt}                                 
  {3pt}                                 
  {\itshape}                            
  {}                                    
  {}                                    
  {\textbf.}                            
  {.5em}                                
  {\thmnote{#3}}                        

\theoremstyle{citing}
\newtheorem*{citing}{}

\newtheoremstyle{citing_definition}     
  {3pt}                                 
  {3pt}                                 
  {}                                    
  {}                                    
  {}                                    
  {\textbf.}                            
  {.5em}                                
  {\thmnote{#3}}                        

\theoremstyle{citing_definition}
\newtheorem*{citing_definition}{}

\DeclareMathOperator{\diam}{diam}       
\DeclareMathOperator{\card}{card}       
\DeclareMathOperator{\dmn}{dmn}         
\DeclareMathOperator{\degree}{degree}   
\DeclareMathOperator{\dist}{dist}       
\DeclareMathOperator{\im}{im}           
\DeclareMathOperator{\Hom}{Hom}         
\DeclareMathOperator{\with}{:}          
\DeclareMathOperator{\without}{\sim}    
\DeclareMathOperator{\restrict}{\llcorner}   
\DeclareMathOperator{\spt}{spt}         
\DeclareMathOperator{\Lip}{Lip}         
\DeclareMathOperator{\Der}{D}           
\DeclareMathOperator{\ap}{ap}           
\DeclareMathOperator{\pt}{pt}           

\newcommand{\ud}{\ensuremath{\,\mathrm{d}}}

\title{Pointwise differentiability of higher order for distributions}

\author{Ulrich Menne}

\begin{document}

\maketitle

\begin{abstract}
	For distributions, we build a theory of higher order pointwise
	differentiability comprising, for order zero, Łojasiewicz's notion of
	point value.  Results include Borel regularity of differentials,
	higher order rectifiability of the associated jets, a
	Rademacher-Stepanov type differentiability theorem, and a Lusin type
	approximation.   A substantial part of this development is new also
	for zeroth order.  Moreover,  we establish a Poincaré inequality
	involving the natural norms of negative order of differentiability.
	As a corollary, we characterise pointwise differentiability in terms
	of point values of distributional partial derivatives.
\end{abstract}

\paragraph{MSC-classes 2010} 46F10 (Primary); 26B05, 41A58 (Secondary).


\section{Introduction}

In their fundamental paper~\cite{MR0136849}, Calderón and Zygmund created a
pointwise differentiability theory for functions in Lebesgue spaces and
applied it to the study of \emph{strong solutions} of systems of linear
elliptic equations.  Proceeding to \emph{weak solutions}, one is naturally led
to investigate pointwise differentiability theory of distributions first.  Our
treatment thereof is strongly influenced by the present and future needs of
regularity questions in geometric measure theory discussed towards the end of
this introduction.

Independent of this motivation, our results also shed new light on the
well-established concept, introduced by Łojasiewicz in \cite{MR0087905}
and~\cite{MR0107167}, of point value of a distribution.  The latter occurs for
instance in the multiplication of distributions (see, e.g.,
Łojasiewicz~\cite{MR0087904} and Itano~\cite{MR0209835}), in Fourier series
(see, e.g., Vindas and Estrada~\cite{MR2355012}), in boundary and initial
value problems in partial differential equations (see, e.g.,
Szmydt~\cite{MR0508961} and Walter~\cite{MR0298415}), and in generalised
integrals (see Estrada and Vindas~\cite{MR2952047}).

\emph{Throughout this introduction, we suppose $k$ is an integer, $0 < \alpha
\leq 1$, $k+\alpha \geq 0$, $n$ is a positive integer, $Y$ is a Banach space,
and $T \in \mathscr D' ( \mathbf R^n, Y )$.}  Our chief concern is the case $Y
= \mathbf R$; hence, separability of~$Y$ is hypothesised whenever convenient
and separability of~$Y^\ast$ is assumed when it may not be omitted.

\subsection{Differentiability theory and Łojasiewicz's point values}

The definition of pointwise differentiability adapts the approach of
Rešetnjak, to transform to all objects to the unit ball, from functions to
distributions (see \cite[p.\,294]{MR0225159-english}).%
\begin{footnote}
	{The Russian original is~\cite{MR0225159}.}
\end{footnote}%
To formulate it, we recall that $R ( \phi )$ is alternatively denoted by $R_x
( \phi(x))$ whenever $R \in \mathscr D' ( \mathbf R^n, Y )$ and $\phi \in
\mathscr D ( \mathbf R^n, Y )$, that $\langle y, \upsilon \rangle$ indicates
the value $\upsilon(y)$ of the pairing of $y \in Y$ and $\upsilon \in Y^\ast$,
and that $\langle \cdot,\cdot \rangle$ is similarly employed for functions
with values in these spaces.

\begingroup \hypertarget{def:pt_diff}{}
	\begin{citing_definition} [\textbf{Definition}
	(see~\ref{def:pt_diff_integral})]
		Whenever $a \in \mathbf R^n$, the distribution $T$ is termed
		\emph{pointwise differentiable of order~$k$ at~$a$} if and
		only if there exists a polynomial function $P : \mathbf R^n
		\to Y^\ast$ of degree at most~$k$ satisfying
		\begin{equation*}
			\lim_{r \to 0+} r^{-k-n} (T-S)_x \big ( \phi (r^{-1}
			(x-a)) \big) = 0 \quad \text{for $\phi \in \mathscr D
			( \mathbf R^n, Y )$},
		\end{equation*}
		where $S \in \mathscr D' ( \mathbf R^n, Y )$ is defined by $S
		( \phi ) = \int \langle \phi, P \rangle \ud \mathscr L^n$ for
		$\phi \in \mathscr D ( \mathbf R^n, Y )$; here, by convention,
		a polynomial function of degree at most~$-1$ is the zero
		function.  As $P$ is unique (see~\ref{lemma:uniqueness}), we
		may define the \emph{$k$-th order pointwise differential
		of~$T$ at~$a$} by $\pt \Der^k T (a) = \Der^k P (a)$ if $k \geq
		0$.
	\end{citing_definition}
\endgroup

For $k = 0$, this definition yields Łojasiewicz's notion of point value
(see~\ref{remark:lojasiewicz_pt_value}).  More generally, adapting the
terminology of Drožžinov and Zav\cprime jalov for $k \geq 0$, our
condition requires the existence of a ``closed quasi-asymptotic expansion of
order $k$ of length $0$'' at $a$ in terms of $P$, see
\ref{remark:asymptotic_analysis}.  We define \emph{pointwise
differentiability of order~$(k,\alpha)$} by employing the condition
\begin{equation*}
	\limsup_{r \to 0+} r^{-k-\alpha-n} (T-S)_x \big ( \phi (r^{-1} (x-a))
	\big) < \infty \quad \text{for $\phi \in \mathscr D ( \mathbf R^n, Y
	)$}
\end{equation*}
in a similar fashion (see \ref{def:pt_diff_hoelder}).  This extends
Zieleźny's notion of boundedness at a point which is defined for $(k,\alpha) =
(-1,1)$ and $n=1$, see~\ref{remark:zielezny_bounded}.

In the author's view, a theory of higher order pointwise differentiability for
a class of objects should consist of at least four results: Borel regularity
of the differentials, rectifiability of the family of $k$~jets, a
Rademacher-Stepanov type theorem, and a Lusin type approximation theorem by
functions of class~$k$.  Such theories (possibly with Borel regularity
replaced by appropriate measurability) have been developed for approximate
differentiation of functions (successively, by Whitney in~\cite{MR0043878},
Isakov in~\cite{MR897693-english},%
\begin{footnote}
	{The Russian original is~\cite{MR897693}.}
\end{footnote}%
and Liu and Tai in~\cite{MR1272893}), for differentiation in Lebesgue spaces
with respect to $\mathscr L^n$ (by Calderón and Zygmund in~\cite{MR0136849}),
for pointwise differentiation of sets (by the author
in~\cite{MR3936235}), and for approximate differentiation of sets (by
Santilli in~\cite{IUMJ_7645}).

For distributions, a pointwise differentiability theory for zeroth order
(i.e., $k+\alpha = 0$) with almost all four results present was developed in
the special case of distributions on the real line (i.e., $n=1$) and $Y =
\mathbf C$ by Zieleźny in~\cite{MR0113134}.  As we will elaborate below, it
appears rather difficult to extend the method of Zieleźny to general~$n$; in
fact, the study of this generalisation was announced for zeroth order
in~\cite[p.\,27]{MR0113134} but seems not to be available as yet.  Employing
different methods, we are able to obtain the four indicated key results for
general~$n$ and all nonnegative orders in the following
Theorems~\hyperlink{thm:A}A--\hyperlink{thm:D}D.

\begingroup \hypertarget{thm:A}{}
	\begin{citing} [\textbf{Theorem A}
	\textrm{(see~\ref{miniremark:setup} and \ref{thm:borel_derivatives})}]
		Suppose $k \geq 0$, $Y$ is separable, and $A$~is the set of
		points at which $T$ is pointwise differentiable of order~$k$.

		Then, $A$ is a Borel set and, for each $y \in Y$, the function
		mapping $a \in A$ onto the real valued symmetric $k$~linear
		map with value
		\begin{equation*}
			\pt \Der^k T (a) ( v_1, \ldots, v_k ) (y) \quad
			\text{at $(v_1, \ldots, v_k) \in (\mathbf R^n )^k$},
		\end{equation*}
		is a Borel function.  In particular, if $Y^\ast$ is separable,
		$\pt \Der^k T$ is a Borel function.
	\end{citing}
\endgroup

The principal conclusion may alternatively be stated using a natural weak
topology (see~\ref{miniremark:setup}) on the space of $Y^\ast$~valued
symmetric $k$~linear maps on~$( \mathbf R^n )^k$.  A classical example
due to Gelfand (see~\cite[p.\,265]{msb_v46_i2_p235}) shows that, without the
separability hypothesis on~$Y^\ast$, the function $\pt \Der^k T$ may be
$\mathscr L^n \restrict A$~nonmeasurable (see~\ref{remark:gelfands_classic});
in particular, that hypothesis may not be omitted.

\begingroup \hypertarget{thm:B}{}
	\begin{citing} [\textbf{Theorem B} (see \ref{thm:cka})]
		Suppose $A$ is the set of points at which $T$ is pointwise
		differentiable of order~$(k,\alpha)$.

		Then, there exists a sequence of compact subsets $C_j$
		of~$\mathbf R^n$ with $A = \bigcup_{j=1}^\infty C_j$ and, if
		$k \geq 0$, also a sequence of functions $f_j : \mathbf R^n
		\to Y^\ast$ of class~$(k,\alpha)$ satisfying
		\begin{equation*}
			\pt \Der^m T (a) = \Der^m f_j (a) \quad \text{for $a
			\in C_j$ and $m = 0, \ldots, k$}
		\end{equation*}
		whenever $j$ is a positive integer.
	\end{citing}
\endgroup

As in the case of the differentiability theory of Calderón and
Zygmund~\cite[Theorems 8 and~9]{MR0136849}, there is no exceptional set in
this rectifiability result.

\begingroup \hypertarget{thm:C}{}
	\begin{citing} [\textbf{Theorem C} (see \ref{corollary:rademacher})]
		Suppose $Y$ is separable and $A$ is the set of points at which
		$T$~is pointwise differentiable of order~$(k,1)$.

		Then, $T$ is pointwise differentiable of order~$k+1$ at
		$\mathscr L^n$~almost all $a \in A$.
	\end{citing}
\endgroup

This theorem in particular proves the existence of point values at $\mathscr
L^n$ almost all points at which the distribution is bounded in the sense of
Zieleźny provided the latter concept is analogously extended to general~$n$,
see \ref{remark:zielezny_bounded} and~\ref{remark:zielezny_rademacher}.

\begingroup \hypertarget{thm:D}{}
	\begin{citing} [\textbf{Theorem D} (see
	\ref{thm:k_lusin_approximation})]
		Suppose $k \geq 0$, $Y^\ast$ is separable, and $A$ is the set
		of points at which $T$ is pointwise differentiable of
		order~$k$.

		Then, for each $\epsilon > 0$, there exists $g : \mathbf R^n
		\to Y^\ast$ of class~$k$ such that
		\begin{equation*}
			\mathscr L^n \big ( A \without \{ a \with \textup{$\pt
			\Der^m T (a) = \Der^m g (a)$ for $m = 0, \ldots,
			k$} \} \big ) < \epsilon.
		\end{equation*}
	\end{citing}
\endgroup

The possible $\mathscr L^n \restrict A$ nonmeasurability of $\pt \Der^k T$ for
nonseparable~$Y^\ast$ shows that one cannot replace $Y^\ast$ by $Y$ in the
separability hypothesis (see~\ref{remark:k_lusin_approximation}).  If $k=0$,
then Theorem~\hyperlink{thm:D}D follows from Theorem~\hyperlink{thm:A}A and
Lusin's theorem.  If $k\geq 1$, then the weaker statement resulting from
replacing $k$ by~$k-1$ in the conclusion of Theorem~\hyperlink{thm:D}{D}
follows, at least if $Y = \mathbf R$, from Theorems \hyperlink{thm:B}{B}
and~\hyperlink{thm:C}{C}, and Whitney's Lusin type approximation result for
functions of class~$(k-1,1)$ by functions of class~$k$, see~\cite[Theorem
4]{MR0043878}.  Accordingly, the main additional information contained in
Theorem~\hyperlink{thm:D}D is the equality of the $k$-th derivatives involved.

Beyond the above four basic properties, the pointwise differentiability theory
for distributions allows to deduce, for nonnegative integers~$l$,
differentiability information of order $k+l$ from differentiability
information of order~$l$ of the $k$-th order distributional derivatives.  This
property is shared by Calderón and Zygmund's theory of differentiation in
Lebesgue spaces with respect to~$\mathscr L^n$, see
\ref{remark:cz_diff_diff}.  However, it fails for approximate differentiation
as Kohn's example in~\cite{MR0427559} shows.  Recalling, for use with $m$-th
order partial derivatives, that $\boldsymbol \Xi (n,m)$ denotes the set of all
$n$~termed sequences of nonnegative integers whose sum equals $m$, our next
theorem formulates the property in question for distributions.

\begingroup \hypertarget{thm:E}{}
	\begin{citing} [\textbf{Theorem E} (see
	\ref{thm:pt_diff_derivatives}\,\eqref{item:pt_diff_derivatives:top_integral})]
		Suppose $k \geq 1$, $l$ is nonnegative integer, $a \in \mathbf
		R^n$, and $\Der^\xi T$ is pointwise differentiable of
		order~$l$ at~$a$ for $\xi \in \boldsymbol \Xi (n,k)$.

		Then, $T$ is pointwise differentiable of order~$k+l$ at~$a$.
	\end{citing}
\endgroup

As the converse is elementary (see~\ref{remark:distributional_derivative}),
Theorem~\hyperlink{thm:E}E in particular yields a natural characterisation of
pointwise differentiability of order~$k$ in terms of the existence of point
values in the sense of Łojasiewicz of the $k$-th order partial derivatives
(see~\ref{corollary:characterisation_pt_diff}).  For general~$n$ and $Y =
\mathbf C$, it was previously only known that the existence of point values
for the partial derivatives implies the existence of point values for the
distribution itself (see Itano~\cite[Lemma~3]{MR0209835}).  In the subcase
$n=1$, a related characterisation of point valued by means of ``integrable''
distributional derivative was obtained by Kim
in~\cite[Proposition~2]{MR3261225}.

For pointwise differentiability of order~$(k+l,\alpha)$, the corresponding
theorem reads as follows.

\begingroup \hypertarget{thm:E'}{}
	\begin{citing} [\textbf{Theorem E$'$} (see
	\ref{thm:pt_diff_derivatives}\,\eqref{item:pt_diff_derivatives:top_hoelder})]
		Suppose $k \geq 1$, $l$ is an integer, $l+\alpha \geq 0$, $a
		\in \mathbf R^n$, and $\Der^\xi T$ is pointwise differentiable
		of order~$(l,\alpha)$ at~$a$ for $\xi \in \boldsymbol \Xi
		(n,k)$.

		Then, $T$ is pointwise differentiable of order~$(k+l,\alpha)$
		at~$a$.
	\end{citing}
\endgroup

\subsection{Concept of the proofs and a Poincaré inequality}

A short proof of a variant of the four basic properties of a differentiability
theory for the simplest case of pointwise (i.e., Peano type) differentiability
of functions in the sense of \cite[2.6, 2.7]{MR3936235} is available
in~\cite[4.6]{MR3936235}.  Here, we mainly focus on the aspects specific to
the setting of distributions.  For this purpose, we recall that $\mathbf
B(a,r)$ denotes the closed ball with centre $a$ and radius $r$, that
\begin{equation*}
	\mathscr D_K ( \mathbf R^n, Y ) = \mathscr D ( \mathbf R^n, Y ) \cap
	\{ \phi \with \spt \phi \subset K \}
\end{equation*}
whenever $K$ is a compact subset of~$\mathbf R^n$, and that, for every
nonnegative integer~$i$, the seminorms $\boldsymbol \nu_{\mathbf
B(0,1)}^i$ defined on $\mathscr E ( \mathbf R^n, Y )$ have value
(see~\ref{miniremark:lip})
\begin{equation*}
	\sup \{ \| \Der^m \phi (x) \| \with x \in \mathbf B(0,1), m = 0,
	\ldots, i \} = \sup \im \| \Der^i \phi \|
\end{equation*}
at $\phi \in \mathscr D_{\mathbf B(0,1)} ( \mathbf R^n, Y)$.  Accordingly,
whenever $\nu$ a seminorm whose restriction to $\mathscr D_{\mathbf B(0,1)} (
\mathbf R^n, Y)$ is a norm on $\mathscr D_{\mathbf B(0,1)} ( \mathbf R^n,
Y)$, we define (see \ref{def:nu_pt_diff_integral}
and~\ref{def:nu_pt_diff_hoelder}) the notions of \emph{$\nu$~pointwise
differentiability of order~$k$} and \emph{$\nu$~pointwise differentiability of
order~$(k,\alpha)$} by requiring that the limit conditions in the
corresponding definitions without prefix are satisfied uniformly for $\phi \in
\mathscr D_{\mathbf B(0,1)} ( \mathbf R^n, Y )$ with $\nu (\phi) \leq 1$.

Then, pointwise differentiability of order~$(k,\alpha)$ is equivalent to
$\boldsymbol \nu_{\mathbf B(0,1)}^i$~pointwise differentiability of
order~$(k,\alpha)$ for some nonnegative integer~$i$,
see~\ref{remark:norm_differentiability}.  The corresponding statement for
pointwise differentiability of order~$k$ holds if and only if $\dim Y <
\infty$, see \ref{remark:dense_diff_crit}
and~\ref{remark:dim_finite_not_omittable}.  With these two facts at hand,
Theorems \hyperlink{thm:A}A and~\hyperlink{thm:B}B are derived as in the case
of functions; that is, employing basic descriptive set theory for
Theorem~\hyperlink{thm:A}A and Whitney's extension theorem for
Theorem~\hyperlink{thm:B}B.

The key to prove Theorems \hyperlink{thm:C}C and~\hyperlink{thm:D}D is the
following theorem which corresponds to~\cite[4.4]{MR3936235} in the case of
functions.  The pattern of proof however follows \cite[Appendix]{MR3023856}
where the case $i=1$, $k=0$, and $\dim Y < \infty$ was treated by means of the
Whitney type partition of unity in~\cite[3.1.13]{MR41:1976}.

\begingroup \hypertarget{thm:F}{}
	\begin{citing} [\textbf{Theorem F} (see~\ref{thm:big_O_little_o})]
		Suppose $i$ is a nonnegative integer, $Y$ is separable, and
		$A$ is the set of points $a \in \mathbf R^n$ at which $T$~is
		$\boldsymbol \nu_{\mathbf B (0,1)}^i$~pointwise differentiable
		of order~$(k-1,1)$ and $\pt \Der^m T(a) = 0$ for $m = 0,
		\ldots, k-1$.

		Then, $\mathscr L^n$~almost all $a \in A$ satisfy the
		following three statements:
		\begin{enumerate}
			\item  The distribution~$T$ is pointwise
			differentiable of order~$k$ at~$a$.
			\item  If $k>0$ or $Y^\ast$ is separable, then $T$ is
			$\boldsymbol \nu_{\mathbf B(0,1)}^i$ pointwise
			differentiable of order~$k$ at~$a$.
			\item If $k>0$, then $\pt \Der^k T(a)=0$.
		\end{enumerate}
	\end{citing}
\endgroup

The hypotheses in the last two items may not be omitted (see
\ref{remark:big_O_little_o:sep} and~\ref{remark:big_O_little_o:zero_order}).
For separable~$Y^\ast$, Theorem~\hyperlink{thm:F}F also yields a version of
the Rademacher-Stepanov type theorem, Theorem~\hyperlink{thm:C}C, for
$\boldsymbol \nu_{\mathbf B(0,1)}^i$~pointwise differentiability
(see~\ref{thm:rademacher_T}\,\eqref{item:rademacher_T:dual_sep}).

The key to establish Theorems \hyperlink{thm:E}E and~\hyperlink{thm:E'}{E$'$},
as well as their versions for $\boldsymbol \nu_{\mathbf B(0,1)}^i$ pointwise
differentiability
(see~\ref{thm:pt_diff_derivatives}\,\eqref{item:pt_diff_derivatives:integral}\,\eqref{item:pt_diff_derivatives:hoelder}),
is the following Poincaré inequality.  Despite the natural significance of a
Poincaré inequality, little appears to be known on such inequalities when
norms of negative order of differentiability are employed.  For our purposes,
the following theorem is sufficient.

\begingroup \hypertarget{thm:G}{}
	\begin{citing} [\textbf{Theorem G}
	(see~\ref{thm:poincare_inequality})]
		Suppose $i$ is a nonnegative integer, $k \geq 1$, $0 \leq
		\kappa < \infty$, and
		\begin{equation*}
			| (\Der^o T) ( \phi ) | \leq \kappa \sup \im \| \Der^i
			\phi \| \quad \text{for $\phi \in \mathscr D_{\mathbf
			B(0,2)} ( \mathbf R^n, Y)$, and $o \in \boldsymbol
			\Xi (n,k)$}.
		\end{equation*}

		Then, there exists a polynomial function $P : \mathbf R^n \to
		Y^\ast$ is of degree at most~$k-1$ such that, for $m = 0,
		\ldots, k-1$ and $\xi \in \boldsymbol \Xi (n,m)$, there holds
		\begin{equation*}
			\big | ( \Der^\xi T ) ( \theta ) - {\textstyle\int}
			\langle \theta, \Der^\xi P \rangle \ud \mathscr L^n
			\big | \leq \Gamma \kappa \sup \im \| \Der^i \theta \|
			\quad \text{for $\theta \in \mathscr D_{\mathbf
			B(0,1)} ( \mathbf R^n, Y )$},
		\end{equation*}
		where $0 \leq \Gamma < \infty$ is determined by $i$, $k$,
		and~$n$.
	\end{citing}
\endgroup

If $k=1$, the proof is carried out by comparing $T$ firstly to a convolution
of~$T$ with a suitable $\Phi \in \mathscr D ( \mathbf R^n, \mathbf R )$ and
subsequently to the value at~$0$ of the function representing that
convolution.  The cases $k>1$ then follow inductively.

There seems to be ample room for further development here.  For instance, one
may wish to study possible subsequent embedding results leading to
strengthenings of the estimate in the conclusion.  An intriguing example of
such an improvement (under supplementary hypotheses not available in the
present circumstances) was given by Allard in~\cite[\S\,1]{MR840267},
see~\ref{remark:allards_strong_constancy}.

Finally, we elaborate on the difference in perspectives allowing us to prove
Theorems \hyperlink{thm:A}A--\hyperlink{thm:D}D, despite previous results did
not go beyond the case $n=1$.  Traditionally, distributions $T \in \mathscr D'
( \mathbf R^n, Y )$ are often considered as generalised functions from
$\mathbf R^n$ into $Y^\ast$.  Accordingly, one strives for a local
representation of the distribution as high order distributional partial
derivative of a function exhibiting a corresponding asymptotic behaviour at a
given point (see Łojasiewicz \cite[4.2 Théorème 1$'$]{MR0107167} for general
$n$ but $k=0$, and Pilipović, Stanković, and Takači \cite[Chapter 2, Theorem
9.2]{MR1062876}%
\begin{footnote}%
	{The condition $F^{(m)} = f$ is missing in the conclusion of the
	theorem but occurs in its proof.}
\end{footnote}%
for general asymptotics but $n=1$).  However, to be applicable for the
present purposes, such a representation would not only need to be available
for general $k$ and $n$ but, in fact, \emph{simultaneously} encode the
differentiability properties of the distribution at uncountably many points.
The latter part of such an endeavour appears unnecessarily difficult to us.
(The case $n=1$ is special as representations then are essentially unique.)
Instead, we always treat $T \in \mathscr D' ( \mathbf R^n, Y )$ as continuous
linear functional $T : \mathscr D ( \mathbf R^n, Y ) \to \mathbf R$ and
accordingly, avoiding any representation, just employ the norms dual to
$\boldsymbol \nu_{\mathbf B(0,1)}^i | \mathscr D_{\mathbf B (0,1)} ( \mathbf
R^n, Y )$ in the relevant estimates.  This difference in perspectives is also
reflected in our notation.

\subsection{Envisaged future developments}

\subsubsection*{Geometric measure theory}

The utility to consider the validity of Theorem~\hyperlink{thm:D}D became
apparent during the author's ongoing investigation of a special case of the
varifold regularity problem formulated jointly with Scharrer
in~\cite[Question~3]{arXiv:1709.05504v1}.  Furthermore, the present paper is
the third in a sequence of studies (initiated by the author
in~\cite{MR3936235} and continued by Santilli in~\cite{IUMJ_7645})
that is ultimately directed towards possible higher order pointwise
differentiability properties of stationary integral varifolds.  For
approximate differentiability of second order, the central elliptic partial
differential equation involves as inhomogeneous term precisely a distribution,
that is $\boldsymbol \nu_{\mathbf B(0,1)}^1$~pointwise differentiable of
order~$0$ at all points in a set, that is compact and has positive $\mathscr
L^n$~measure but does not possess further regularity properties
(see~\cite[4.4\,(6)]{MR3023856}).

\subsubsection*{Elliptic partial differential equations}

From the point of view of elliptic partial differential equations, it is
natural to aim to replace the norms~$\boldsymbol \nu_{\mathbf B(0,1)}^i$ in
the present theory by the norms $\nu_{i,p}$ defined by
\begin{equation*}
	\nu_{i,p} ( \phi ) = \big ( {\textstyle\int \| \Der^i \phi \|^p \ud
	\mathscr L^n } \big )^{1/p} \quad \text{for $\phi \in \mathscr
	D_{\mathbf B(0,1)} ( \mathbf R^n, Y )$},
\end{equation*}
whenever $i$ is a nonnegative integer and $1 < p < \infty$.  A
Rademacher-Stepanov type result for these norms with $(k,\alpha) = (-1,1)$,
$i=1$, and $\dim Y < \infty$ was proven by the author
in~\cite[3.13]{MR3023856}.  Furthermore, initial elements of the corresponding
pointwise differentiability theory for weak solutions of (linear and
non-linear) elliptic partial differential equations were provided by him in
\cite[8.4]{MR2898736} and~\cite[3.11, 3.18]{MR3023856}.  Extending these
results to a more complete theory modelled upon that of Calderón and
Zygmund~\cite{MR0136849}, but including certain non-linear equations, appears
natural not only as development within elliptic partial differential equations
but also as case study for the afore-mentioned regularity questions in
geometric measure theory.

\subsubsection*{Theory of distributions}

Following a different line of thought, one might strive to determine optimal
conditions on the Banach space~$Y$ in the spirit of the Radon-Nikodým property
(see~\cite[\S\,5]{MR1727673}) for the validity of the present theory.
Finally, Yoshinaga's reduction (see~\cite[p.\,24]{MR0226400}) of Łojasiewicz's
concept of fixation of variables (see~\cite{MR0107167}) to point values with
$Y = \mathscr D_K ( \mathbf R^m, \mathbf R )$ suggests to additionally include
certain locally convex spaces~$Y$ in the study of possible extensions.

\subsection{Acknowledgement}

The author is grateful to Professor Guido De Philippis for a discussion on
Allard's strong constancy lemma (see \ref{remark:allards_strong_constancy}),
to Professor Bernd Kirchheim for a discussion leading him to find
Łojasiewicz's fundamental papers \cite{MR0087905} and~\cite{MR0107167} to
Professor Ricardo Estrada for pointing out the connection to asymptotic
expansions (see in particular \ref{remark:asymptotic_analysis} and
\ref{def:nu_pt_diff_integral}) and for an alternative idea to prove Theorem
\hyperlink{thm:E}E, see \ref{remark:Estrada}, and to Dr~Sławomir Kolasiński
for his comments on this manuscript.  The initial version of this paper (see
\href{https://arxiv.org/abs/1803.10855v1}{\path{arXiv:1803.10855v1}}) was
written while the author worked at the University of Leipzig and the Max
Planck Institute for Mathematics in the Sciences.

\subsection{Notation}

As a rule, our terminology is that of~\cite[pp.\,669--676]{MR41:1976}; in
particular, the field involved in vectorspaces and linear maps is considered
to be~$\mathbf R$ by default, whenever $\mu$ measures $X$ and $Y$ is a
topological space, a $Y$ valued $\mu$ measurable functions means a function
$f$ with $\dmn f \subset X$, $\im f \subset Y$, and $\mu ( X \without \dmn f )
= 0$ such that $f^{-1} [ V ]$ is $\mu$ measurable, whenever $V$ is an open
subset of $Y$, and
\begin{align*}
	\mu_{(p)} (f) & = \big ( {\textstyle \int} |f|^p \ud \mu \big
	)^{1/p} && \text{if $p < \infty$}, \\
	\mu_{(\infty)} (f) & = \inf \big \{ s \with s \geq 0, \mu \, \{ x
	\with |f(x)|>s \} = 0 \big \} && \text{if $p = \infty$},
\end{align*}
whenever $\mu$ measures $X$, $Z$ is a separable Banach space, $f$ is a
$Z$ valued $\mu$~measurable function, and $1 \leq p \leq \infty$.  The
only exception to this rule is that we employ the more common locally convex
topology on~$\mathscr D ( \mathbf R^n, Y)$ defined for instance
in~\cite[2.13]{MR3528825}; this modified topology however leads to the
same dual space $\mathscr D' ( \mathbf R^n, Y)$ of continuous linear maps $T :
\mathscr D ( \mathbf R^n, Y) \to \mathbf R$ as in
\cite[4.1.1]{MR41:1976}, see \cite[2.17\,(1)]{MR3528825}.
Finally, we additionally employ the term \emph{function of class~$(k,\alpha)$}
defined for $k \geq 0$ in~\cite[2.4]{MR3936235}.

For the convenience of the reader, we next review some of the basic
terminology from multilinear algebra related to our treatment of polynomial
functions (see~\cite[\S\,1.9, \S\,1.10]{MR41:1976}).  Whenever $Z$ is a normed
space and $k \geq 1$, $\bigodot^k ( \mathbf R^n, Z )$ denotes the normed space
of $k$~linear symmetric maps of $( \mathbf R^n )^k$ into~$Z$ with
\begin{equation*}
	\| \phi \| = \sup \big \{ |\phi (v_1, \ldots, v_k)| \with \text{$v_m
	\in \mathbf R^n, |v_m| \leq 1$ for $m=1, \ldots, k$} \big \} \quad
\end{equation*}
for $\phi \in \bigodot^k ( \mathbf R^n, Z )$.  Moreover, $\bigodot^0 (
\mathbf R^n, Z ) = Z$.  The symmetric algebra of $\mathbf R^n$, given by
\begin{equation*}
	{\textstyle \bigodot_\ast \mathbf R^n = \bigoplus_{m=0}^\infty
	\bigodot_m \mathbf R^n},
\end{equation*}
is a commutative associative graded algebra with unit element $1 \in \mathbf R
= \bigodot_0 \mathbf R^n$ and $\odot$ denotes its multiplication.  Whenever
$e_1, \ldots, e_n$ form a basis of~$\mathbf R^n$, the vectors $e^\xi =
(e_1)^{\xi_1} \odot \cdots \odot (e_n)^{\xi_n}$ corresponding to $\xi \in
\boldsymbol \Xi (n,m)$ form a basis of~$\bigodot_m \mathbf R^n$.  Moreover,
the canonical linear isomorphism
\begin{equation*}
	{\textstyle \Hom ( \bigodot_m \mathbf R^n, Z ) \simeq \bigodot^m (
	\mathbf R^n, Z )}
\end{equation*}
maps $h \in \Hom ( \bigodot_m \mathbf R^n, Z )$ onto $\psi \in \bigodot^m (
\mathbf R^n, Z)$ satisfying
\begin{equation*}
	\psi (v_1, \ldots, v_m) = h ( v_1 \odot \cdots \odot v_m ) \quad
	\text{for $v_1, \ldots, v_m \in \mathbf R^n$}.
\end{equation*}
Accordingly, one alternately denotes $h(\eta)$ by $\langle \eta, \psi \rangle$
for $\eta \in \bigodot_m \mathbf R^n$.  The interior multiplications
$\mathop{\lrcorner} : \bigodot_m \mathbf R^n \times \bigodot^k ( \mathbf R^n,
Z ) \to \bigodot^{k-m} ( \mathbf R^n, Z )$ corresponding to $m = 0, \ldots, k$
are characterised by
\begin{equation*}
	\langle \zeta , \eta \mathop{\lrcorner} \psi \rangle = \langle \zeta
	\odot \eta, \psi \rangle \quad {\textstyle\text{for $\zeta \in
	\bigodot_{k-m} \mathbf R^n$, $\eta \in \bigodot_m \mathbf R^n$, and
	$\psi \in \bigodot^k ( \mathbf R^n, Z)$}}.
\end{equation*}
Finally, the norm on~$\bigodot_m \mathbf R^n$ is defined by
\begin{equation*}
	{\textstyle \| \eta \| = \sup \big \{ \langle \eta, \psi \rangle \with
	\psi \in \bigodot^m (\mathbf R^n, \mathbf R), \| \psi \| \leq 1 \big
	\} \quad \text{whenever $\eta \in \bigodot_m \mathbf R^n$}}.
\end{equation*}

\section{Basic properties}

In the present section, we firstly collect the necessary functional analytic
preliminaries in~\ref{def:Y_topology}--\ref{miniremark:setup}.  Then, we
formally introduce our definitions of pointwise differentiability
in~\ref{lemma:uniqueness}--\ref{remark:origin_resetnjak}.  Finally, we derive
basic properties of these concepts along with four examples
in~\ref{lemma:dense_diff_crit}--\ref{example:distribution_representable_by_integration}.

\begin{definition} [see \protect{\cite[p.~420]{MR0117523}}]
\label{def:Y_topology}
	Suppose $Y$ is a Banach space.

	Then, the topology on $Y^\ast$ inherited from~$\mathbf R^Y$ is termed
	the \emph{$Y$~topology}.
\end{definition}

\begin{miniremark} \label{miniremark:lebesgue_duality}
	Suppose $\mu$ measures $X$, the set $X$ may be expressed as union
	of a countable family whose members are $\mu$ measurable sets with
	finite $\mu$ measure, $Y$ is a separable Banach space, and
	\begin{equation*}
		H : \mathbf L_1 ( \mu, Y ) \cap \{ \theta \with \dmn \theta =
		X \} \to \mathbf R
	\end{equation*}
	is a linear homomorphism for which $M = \sup \{ | H(\theta) | \with
	\mu_{(1)} (\theta) \leq 1 \} < \infty$.%
	\begin{footnote}%
		{The intersection with $\{ \theta \with \dmn \theta = X
		\}$ may not be omitted as, with the definition given in
		\cite[2.4.12]{MR41:1976}, $\mathbf L_1 ( \mu, Y )$
		may fail to be a vector space as no $\theta$ whose domain is a
		proper subset of $X$ has an additive inverse in $\mathbf L_1 (
		\mu, Y)$.  Similar remarks hold for $\mathbf L_p ( \mu, Y)$
		and $\mathbf A ( \mu, Y )$.}
	\end{footnote}%
	Then, there exists a $\mu$~almost unique, $Y^\ast$~valued function $g$
	that is $\mu$ measurable with respect to the $Y$~topology and
	satisfies $\mu_{(\infty)} ( \| g \|) = M$ and
	\begin{equation*}
		H (\theta) = {\textstyle\int \langle \theta, g \rangle \ud \mu
		} \quad \text{whenever $\theta \in \mathbf L_1 ( \mu, Y )$ and
		$\dmn \theta = X$};
	\end{equation*}
	in fact, this is a special case of \cite[Chapter~7,
	Section~4]{MR0276438}.%
	\begin{footnote}
		{More elementary, one may pass from the case $Y = \mathbf R$
		treated in~\cite[2.5.7\,(ii)]{MR41:1976} to the general case
		adapting the method of~\cite[2.5.12]{MR41:1976}; this is
		carried out in~\cite[4.6.3~2]{snulmenn:diploma_thesis}.}
	\end{footnote}%
	In \ref{example:nonmeasurable}, an example showing that $g$
	may fail to be $\mu$ measurable with respect to the norm topology on
	$Y$ will be provided.
\end{miniremark}

\begin{definition}
	Suppose $Y$ and $Z$ are normed spaces.

	Then, we define $\| \tau \|$ for $\tau \in Y \otimes Z$ to be the
	infimum of the set of numbers
	\begin{equation*}
		\sum_{i=1}^N | y_i | \, | z_i |
	\end{equation*}
	corresponding to all positive integers $N$ and $y_i \in Y$, $z_i \in
	Z$ for $i = 1, \ldots, N$ satisfying $\tau = \sum_{i=1}^N y_i \otimes
	z_i$.
\end{definition}

\begin{remark} \label{remark:projective_norm}
	As in~\cite[Proposition 2.1, Theorem 2.9]{MR1888309}, where the
	slightly more elaborate case of Banach spaces is treated, we see that
	the function $\| \cdot \|$ is a norm on~$Y \otimes Z$ inducing a
	linear isometry of $( Y \otimes Z)^\ast$ with the space of continuous
	bilinear maps from $Y \times Z$ into~$\mathbf R$.
\end{remark}

\begin{remark} \label{remark:tensor_sum}
	As in the case of vectorspaces (see~\cite[1.1.2]{MR41:1976}), if $Z_1$
	and $Z_2$ are normed spaces, then $Y \otimes ( Z_1
	\oplus Z_2 ) \simeq ( Y \otimes Z_1 ) \oplus (
	Y \otimes Z_2 )$ as normed spaces.
\end{remark}

\begin{remark} \label{remark:tensor_completeness}
	If $Y$ is complete and $\dim Z < \infty$, then $Y \otimes Z$ is
	complete by~\ref{remark:tensor_sum}.
\end{remark}

\begin{remark} \label{remark:mass_norm}
	The norm $\| \cdot \|$ on $\bigodot_m \mathbf R^n$ renders the
	canonical linear isomorphism $\Hom ( \bigodot_m \mathbf R^n, Z )
	\simeq \bigodot^m ( \mathbf R^n, Z )$ an isometry whenever $Z$ is a
	normed space; in fact, this follows as in~\ref{remark:projective_norm}
	from the characterisation of~$\| \cdot \|$
	in~\cite[1.10.5]{MR41:1976}.
\end{remark}

\begin{miniremark} \label{miniremark:setup}
	Suppose $Y$ is a Banach space and $Z_m = \bigodot^m ( \mathbf R^n,
	Y^\ast)$ whenever $m$ is a nonnegative integer.  Then, the normed
	space $Y \otimes \bigodot_m \mathbf R^n$ is complete
	by~\ref{remark:tensor_completeness} and we will employ the canonical
	linear isometry (see \ref{remark:projective_norm}
	and~\ref{remark:mass_norm})
	\begin{equation*}
		{\textstyle Z_m \simeq ( Y \otimes \bigodot_m \mathbf R^n
		)^\ast}
	\end{equation*}
	and the \emph{$Y \otimes \bigodot_m \mathbf R^n$~topology on~$Z_m$}
	defined by requiring this isometry to be a homeomorphism with the $Y
	\otimes \bigodot_m \mathbf R^n$ topology on $( Y \otimes \bigodot_m
	\mathbf R^n )^\ast$.  The compact topology inherited from this
	topology by $Z_m \cap \{ \psi \with \| \psi \| \leq 1 \}$ is
	metrisable if and only if $Y$ is separable (see
	\ref{remark:tensor_sum} and~\cite[V.4.2, V.5.1]{MR0117523}).  In this
	case, the class of Borel sets of the $Y \otimes \bigodot_m \mathbf
	R^n$ topology on $Z_m$ is generated by the family consisting of all
	sets
	\begin{equation*}
		Z_m \cap \{ \psi \with \langle y \otimes \eta, \psi \rangle <
		\kappa \}
	\end{equation*}
	corresponding to $y \in Y$, $\eta \in \bigodot_m \mathbf R^n$, and
	$\kappa \in \mathbf R$; in fact, as $Y \otimes \bigodot_m \mathbf R^n$
	is spanned by such $y \otimes \eta$, one may
	apply~\cite[2.22]{MR3528825}.%
	\begin{footnote}
		{More elementary, one may recall the fifth paragraph of the
		proof of~\cite[2.5.12]{MR41:1976}.}
	\end{footnote}%
	If $Y^\ast$ is separable in its norm topology, then so is $Z_m$ and
	hence the classes of Borel sets in $Z_m$ with respect to the $Y
	\otimes \bigodot_m \mathbf R^n$~topology and the norm topology agree.
\end{miniremark}

\begin{lemma} \label{lemma:uniqueness}
	Suppose $k$ is a nonnegative integer, $P : \mathbf R^n \to \mathbf R$
	is a polynomial function, $\Delta$ is a dense subset of $\mathscr D (
	\mathbf R^n, \mathbf R )$, $a \in \mathbf R^n$, and
	\begin{equation*}
		\lim_{r \to 0+} r^{-k-n} {\textstyle\int \phi ( r^{-1} (x-a) )
		P (x) \ud \mathscr L^n \, x } = 0 \quad \text{for $\phi \in
		\Delta$}.
	\end{equation*}

	Then, $\Der^m P (a) = 0$ for $m=0, \ldots, k$.
\end{lemma}

\begin{proof}
	Possibly replacing $P(x)$ by $\sum_{m=0}^k \langle (x-a)^m/m!, \Der^m
	P(a) \rangle$, we assume $\degree P \leq k$.  We let $Z$ denote the
	vectorspace of all polynomial functions $Q : \mathbf R^n \to \mathbf
	R$ of degree at most~$k$ endowed with the norm whose value at $Q$
	equals $\sum_{m=0}^k \| \Der^m Q (0) \|$.  Since $\dim Z < \infty$,
	\cite[Chapter 1, p.~13, Theorem~2]{MR910295} yields that $Z$ may be
	homeomorphically embedded into $\mathbf R^\Delta$ by associating to $Q
	\in Z$ the function with value $\int \phi \, Q \ud \mathscr L^n$ at
	$\phi \in \Delta$.  Hence, as
	\begin{equation*}
		r^{-k-n} {\textstyle\int \phi ( r^{-1} (x-a) ) P (x) \ud
		\mathscr L^n \, x } = {\textstyle\int} \phi (x) P_r (x) \ud
		\mathscr L^n \, x,
	\end{equation*}
	where $P_r (x) = r^{-k} P(a+rx)$, we conclude $P_r \to 0$ as $r \to
	0+$ and $P = 0$.
\end{proof}

\begin{definition} \label{def:pt_diff_integral}
	Suppose $k$~is an integer with $k \geq -1$, $Y$ is a Banach space, $T
	\in \mathscr D' ( \mathbf R^n, Y )$, and $a \in \mathbf R^n$.

	Then, $T$ is termed \emph{pointwise differentiable of order~$k$
	at~$a$} if and only if there exists a polynomial function $P : \mathbf
	R^n \to Y^\ast$ of degree at most~$k$ satisfying
	\begin{equation*}
		\lim_{r \to 0+} r^{-k-n} (T-S)_x \big ( \phi (r^{-1} (x-a))
		\big) = 0 \quad \text{for $\phi \in \mathscr D ( \mathbf R^n,
		Y )$},
	\end{equation*}
	where $S \in \mathscr D' ( \mathbf R^n, Y )$ is defined by $S ( \phi )
	= \int \langle \phi, P \rangle \ud \mathscr L^n$ for $\phi \in
	\mathscr D ( \mathbf R^n, Y )$; here, by convention, a polynomial
	function of degree at most~$-1$ is the zero function.  As $P$ is
	unique by~\ref{lemma:uniqueness}, we term $S$ the \emph{$k$~jet of $T$
	at~$a$} and, if $k \geq 0$, we define the \emph{$k$-th order pointwise
	differential of~$T$ at~$a$} by  $\pt \Der^k T (a) = \Der^k P (a)$.
\end{definition}

\begin{remark} \label{remark:pt_diff}
	Whenever $m = 0, \ldots, k$, it follows that $T$ is pointwise
	differentiable of order~$m$ and $\pt \Der^m T (a) = \Der^m P (a)$.
\end{remark}

\begin{remark} \label{remark:distributional_derivative}
	If $T$ is pointwise differentiable of order~$k$ at~$a$ and $S$~is the
	$k$~jet of~$T$ at~$a$, then, for $m = 0, \ldots, k$ and $\xi \in
	\boldsymbol \Xi (n,m)$, the distribution $\Der^\xi T$ is pointwise
	differentiable of order~$k-m$ at~$a$, $\Der^\xi S$~is the $k-m$~jet
	of~$\Der^\xi T$ at~$a$ and, denoting by $e_1, \ldots, e_n$ the
	standard basis of~$\mathbf R^n$, we have
	\begin{equation*}
		\pt \Der^\mu (\Der^\xi T) (a) = e^\xi \mathop{\lrcorner} \pt
		\Der^{m+\mu} T (a) \quad \text{for $\mu = 0, \ldots, k-m$},
	\end{equation*}
	where $e^\xi = (e_1)^{\xi_1} \odot \cdots \odot (e_n)^{\xi_n}$ and the
	powers are computed in $\bigodot_\ast \mathbf R^n$.  The converse type
	of implication will be treated in
	\ref{thm:pt_diff_derivatives}\,\eqref{item:pt_diff_derivatives:top_integral}
	and~\ref{corollary:characterisation_pt_diff}.
\end{remark}

\begin{remark} \label{remark:lojasiewicz_pt_value}
	For $k = 0$, $Y = \mathbf C$, and $\mathbf C$~linear $T$, this concept
	was introduced in~\cite[\S\,3]{MR0107167}; the subcase $n=1$ had
	appeared before in~\cite[1.1]{MR0087905}.
\end{remark}

\begin{remark} \label{remark:expansion}
	Suppose additionally $k \geq 0$ and $Q : \mathbf R^n \to Y^\ast$
	is a homogeneous polynomial function of degree $k$.  Then, $T$ is
	pointwise differentiable of order $k$ at $a$ with $\pt \Der^k T(a) =
	\Der^k Q (0)$ if and only if $T$ is pointwise differentiable of order
	$k-1$ at $a$ and, denoting the $k-1$ jet of $T$ by $R$,
	\begin{equation*}
		\lim_{r \to 0+} r^{-k-n} (T-R)_x \big ( \phi (r^{-1}(x-a))
		\big ) = {\textstyle\int} \langle \phi, Q \rangle \ud \mathscr
		L^n \quad \text{for $\phi \in \mathscr D ( \mathbf R^n, Y)$}.
	\end{equation*}
\end{remark}

\begin{remark} \label{remark:asymptotic_analysis}
	Suppose $a = 0$.  Then, employing the terminology of \cite[Definition
	1]{MR3429948}, the limit condition in \ref{remark:expansion} yields
	that $T-R$ has ``quasiasymptotics at zero'' on $\mathscr D ( \mathbf
	R^n, Y )$ with respect to $\rho$ along the group of homotheties on
	$\mathbf R^n$, where $\rho : \{ r \with 0 < r < \infty \} \to \mathbf
	R$ satisfies $\rho (r)=r^k$ for $0 < r < \infty$.  Correspondingly
	adapting the terminology of \cite[\S 4,
	Definition]{MR0435696-english},%
	\begin{footnote}%
		{The Russian original is \cite{MR0435696}.}
	\end{footnote}%
	pointwise differentiability of order $k$ at $a$ may be viewed as
	``closed quasi-asymptotic expansion of order $k$ of length $0$'' in
	terms of the $k$ jet of $T$ at zero.
\end{remark}

\begin{definition} \label{def:nu_pt_diff_integral}
	Suppose $k$~is a nonnegative integer, $Y$ is a Banach space, $T \in
	\mathscr D' ( \mathbf R^n, Y )$, $\nu$ is a seminorm whose
	restriction to $\mathscr D_{\mathbf B(0,1)} ( \mathbf R^n, Y )$ is a
	norm on $\mathscr D_{\mathbf B(0,1)} ( \mathbf R^n, Y )$, and $a \in
	\mathbf R^n$.

	Then, $T$ is termed \emph{$\nu$~pointwise differentiable of order~$k$
	at~$a$} if and only if there exists a polynomial function $P : \mathbf
	R^n \to Y^\ast$ of degree at most~$k$ satisfying
	\begin{equation*}
		\lim_{r \to 0+} \sup \big \{ r^{-k-n} (T-S)_x \big ( \phi
		(r^{-1} (x-a)) \big) \with \phi \in \mathscr D_{\mathbf
		B(0,1)} ( \mathbf R^n, Y ), \nu ( \phi ) \leq 1 \big \} = 0,
	\end{equation*}
	where $S \in \mathscr D' ( \mathbf R^n, Y )$ is defined by $S ( \phi )
	= \int \langle \phi, P \rangle \ud \mathscr L^n$ for $\phi \in
	\mathscr D ( \mathbf R^n, Y )$.
\end{definition}

\begin{remark} \label{remark:nu_pt_diff}
	In this case, $T$ is pointwise differentiable of order~$k$ at~$a$ and
	$S$ is the $k$~jet of~$T$ at~$a$.
\end{remark}

\begin{remark}
	In the present paper, this notion will usually be employed with $\nu =
	\boldsymbol \nu_{\mathbf B(0,1)}^i$ for some nonnegative integer~$i$.
	Adding the parameter $1 \leq p \leq \infty$, the family of norms
	$\nu_{i,p}$ on $\mathscr D_{\mathbf B (0,1)} ( \mathbf R^n, Y )$
	defined by
	\begin{equation*}
		\nu_{i,p} ( \phi ) = ( \mathscr L^n )_{(p)} ( \Der^i \phi )
		\quad \text{whenever $\phi \in \mathscr D_{\mathbf B(0,1)} (
		\mathbf R^n, Y )$}
	\end{equation*}
	could be studied; in fact, for the case $k=0$, $n=1$, $Y = \mathbf C$,
	and $\mathbf C$~linear $T$, the subcase $i=0$ and $p = 1$ of the
	concept occurred in~\cite[Lemma~3]{MR0218896} and a basic
	characterisation of the subcase $i=0$ and $p < \infty$ was obtained
	in~\cite[Theorem~10]{MR3261225}.
\end{remark}

\begin{remark} \label{remark:value_of_some_order}
	The case $\nu = \boldsymbol \nu_{\mathbf B(0,1)}^i$, $k=0$, $Y =
	\mathbf C$, and $\mathbf C$~linear~$T$ of this concept was introduced
	as ``value of~$T$ at~$a$ of order not exceeding~$i$''
	in~\cite[\S\,8.2]{MR0107167}; the subcase $n=1$ had appeared before
	in~\cite[\S\,2.4 Définition, \S\,4.4 Théorème]{MR0087905}.
\end{remark}

\begin{definition} \label{def:pt_diff_hoelder}
	Suppose $k$~is an integer, $0 < \alpha \leq 1$, $k+\alpha \geq 0$, $Y$
	is a Banach space, $T \in \mathscr D' ( \mathbf R^n, Y )$, and $a \in
	\mathbf R^n$.

	Then, $T$ is termed \emph{pointwise differentiable of
	order~$(k,\alpha)$ at~$a$} if and only if there exists a polynomial
	function $P : \mathbf R^n \to Y^\ast$ of degree at most~$k$ satisfying
	\begin{equation*}
		\limsup_{r \to 0+} r^{-k-\alpha-n} (T-S)_x \big ( \phi (r^{-1}
		(x-a)) \big) < \infty \quad \text{for $\phi \in \mathscr D (
		\mathbf R^n, Y )$},
	\end{equation*}
	where $S \in \mathscr D' ( \mathbf R^n, Y )$ is defined by $S ( \phi )
	= \int \langle \phi, P \rangle \ud \mathscr L^n$ for $\phi \in
	\mathscr D ( \mathbf R^n, Y )$.
\end{definition}

\begin{remark} \label{remark:uniform_boundedness}
	In this case, $T$ is pointwise differentiable of order~$k$ at~$a$, $S$
	is the $k$~jet of~$T$ at~$a$, and, whenever $K$ is a compact subset
	of~$\mathbf R^n$ and $0 < s < \infty$, there exist a nonnegative
	integer~$i$ and $0 \leq M < \infty$ satisfying
	\begin{equation*}
		\big | r^{-k-\alpha-n} (T-S)_x \big ( \phi (r^{-1} (x-a)) \big
		) \big | \leq M \boldsymbol \nu_K^i ( \phi ) \quad \text{for
		$\phi \in \mathscr D_K ( \mathbf R^n, Y )$}
	\end{equation*}
	whenever $0 < r \leq s$; in fact, defining $R_r : \mathscr D_K (
	\mathbf R^n, Y ) \to \mathbf R$ by
	\begin{equation*}
		R_r ( \phi ) = r^{-k-\alpha-n} ( T-S )_x \big ( \phi ( r^{-1}
		(x-a)) \big ) \quad \text{for $0 < r \leq s$ and $\phi \in
		\mathscr D_K ( \mathbf R^n, Y )$}
	\end{equation*}
	and noting the continuity of~$R_r ( \phi )$ in~$r$, we have
	\begin{equation*}
		\sup \{ |R_r(\phi)| \with 0 < r \leq s \} < \infty \quad
		\text{for $\phi \in \mathscr D_K ( \mathbf R^n, Y )$},
	\end{equation*}
	so that the last part follows from the principle of uniform
	boundedness (see~\cite[II.1.11]{MR0117523}) since $\mathscr D_K (
	\mathbf R^n, Y)$ is an ``$F$-space'' in the terminology
	of~\cite[II.1.10]{MR0117523} (see, e.g., \cite[2.14]{MR3528825} and
	\cite[2.4]{MR3626845}).
\end{remark}

\begin{remark} \label{remark:zielezny_bounded}
	The case $n=1$, $k=-1$, $\alpha=1$, $Y = \mathbf C$, and $\mathbf
	C$~linear $T$ of this concept was introduced in~\cite[\S\,1.1
	Definition]{MR0113134}.
\end{remark}

\begin{definition} \label{def:nu_pt_diff_hoelder}
	Suppose $k$~is an integer, $0 < \alpha \leq 1$, $k+\alpha \geq 0$, $Y$
	is a Banach space, $T \in \mathscr D' ( \mathbf R^n, Y )$, $\nu$
	is a seminorm whose restriction to $\mathscr D_{\mathbf B(0,1)} (
	\mathbf R^n, Y )$ is a norm on $\mathscr D_{\mathbf B(0,1)} ( \mathbf
	R^n, Y )$, and $a \in \mathbf R^n$.

	Then, $T$ is termed \emph{$\nu$~pointwise differentiable of
	order~$(k,\alpha)$ at~$a$} if and only if there exists a polynomial
	function $P : \mathbf R^n \to Y^\ast$ of degree at most~$k$ satisfying
	\begin{multline*}
		\limsup_{r \to 0+} \sup \big \{ r^{-k-\alpha-n} (T-S)_x \big (
		\phi (r^{-1} (x-a)) \big) \with \phi \in \mathscr D_{\mathbf
		B(0,1)} ( \mathbf R^n, Y ), \nu ( \phi ) \leq 1 \big \} \\
		< \infty,
	\end{multline*}
	where $S \in \mathscr D' ( \mathbf R^n, Y )$ is defined by $S ( \phi )
	= \int \langle \phi, P \rangle \ud \mathscr L^n$ for $\phi \in
	\mathscr D ( \mathbf R^n, Y )$.
\end{definition}

\begin{remark} \label{remark:norm_differentiability}
	In this case, $T$ is pointwise differentiable of order~$(k,\alpha)$
	at~$a$ and $S$~is the $k$~jet of~$T$ at~$a$
	by~\ref{remark:uniform_boundedness}.  Hence,
	\ref{remark:uniform_boundedness} also yields the following
	proposition.  \emph{Whenever $k$~is an integer, $0 < \alpha \leq 1$,
	$k+\alpha \geq 0$, $Y$ is Banach space, $T \in \mathscr D' ( \mathbf
	R^n, Y )$, and $a \in \mathbf R^n$, the distribution $T$ is pointwise
	differentiable of order~$(k,\alpha)$ at~$a$ if and only if, for some
	nonnegative integer~$i$, it is $\boldsymbol \nu_{\mathbf B
	(0,1)}^i$~pointwise differentiable of order~$(k,\alpha)$ at~$a$}.  The
	case $n=1$, $k=-1$, $\alpha = 1$, $Y = \mathbf C$, and $\mathbf
	C$~linear $T$ thereof also follows from~\cite[\S\,1.1
	Satz\,1.1]{MR0113134}.
\end{remark}

\begin{remark} \label{remark:bounded_of_some_order}
	According to the characterisation~\cite[Satz 1.5]{MR0113134}, the case
	$n=1$, $k=-1$, $\alpha = 1$, $Y=\mathbf C$, $\mathbf C$~linear $T$,
	and $\nu = \boldsymbol \nu_{\mathbf B (0,1)}^i$ of our concept is
	equivalent to the notion of ``boundedness of order not exceeding~$i$''
	of~\cite[\S\,1.3 Definition]{MR0113134}.
\end{remark}

\begin{remark} \label{remark:origin_resetnjak}
	Our notions of pointwise differentiability of higher order (see
	\ref{def:pt_diff_integral}, \ref{def:nu_pt_diff_integral},
	\ref{def:pt_diff_hoelder}, and \ref{def:nu_pt_diff_hoelder}) are
	adaptations of similar concepts for functions in
	\cite[p.~294]{MR0225159-english} to distributions.
\end{remark}

\begin{definition}
	A subset $A$ of a topological space $X$ is termed \emph{sequentially
	dense} if and only if each point in $X$ is the limit of some sequence
	in $A$.
\end{definition}

\begin{lemma} \label{lemma:dense_diff_crit}
	Suppose $k$~is an integer, $k \geq -1$, $Y$ is a Banach space, $\Delta$
	is a sequentially dense subset of~$\mathscr D ( \mathbf R^n, Y )$, $T
	\in \mathscr D' ( \mathbf R^n, Y )$, $a \in \mathbf R^n$, $T$ is
	pointwise differentiable of order $(k,1)$ at~$a$, and $P : \mathbf R^n
	\to Y^\ast$ is a polynomial function of degree at most~$k+1$
	satisfying
	\begin{equation*}
		\lim_{r \to 0+} r^{-k-1-n} (T-S)_x \big ( \phi ( r^{-1} (x-a)
		) \big ) = 0 \quad \text{for $\phi \in \Delta$},
	\end{equation*}
	where $S \in \mathscr D' ( \mathbf R^n, Y )$ is defined by $S ( \phi )
	= \int \langle \phi, P \rangle \ud \mathscr L^n$ for $\phi \in
	\mathscr D ( \mathbf R^n, Y )$.
	
	Then, $T$~is pointwise differentiable of order~$k+1$ at~$a$ and $S$~is
	the $k+1$~jet of~$T$ at~$a$.  If, additionally, we have that $\dim Y <
	\infty$, that $i$ is a nonnegative integer, and that $T$ is
	$\boldsymbol \nu_{\mathbf B(0,1)}^i$ pointwise differentiable of
	order~$(k,1)$ at $a$, then $T$~is $\boldsymbol \nu_{\mathbf
	B(0,1)}^{i+1}$ pointwise differentiable of order~$k+1$ at~$a$.
\end{lemma}

\begin{proof}
	Suppose $\theta \in \mathscr D ( \mathbf R^n, Y )$.  Taking $\phi_j
	\in \Delta$ with $\phi_j \to \theta$ as $j \to \infty$, there exists
	(see, e.g., \cite[2.14, 2.15]{MR3528825}) a number $0 \leq t < \infty$
	with $\spt \phi_j \subset K$ for every positive integer~$j$,
	where $K = \mathbf R^n \cap \mathbf B ( 0,t )$.  Applying
	\ref{remark:uniform_boundedness} with $\alpha = 1 = s$, we obtain $i$
	and $M$.  Since $\pt \Der^m T (a) = \pt \Der^m S (a)$ for $m = 0,
	\ldots, k$ by \ref{lemma:uniqueness} and
	\begin{equation*}
		{\textstyle \int_{\mathbf B (a,rt)}} |x-a|^{k+1} (k+1)!^{-1}
		\| \Der^{k+1} P ( a ) \| \ud \mathscr L^n \, x = N r^{n+k+1},
	\end{equation*}
	where $N = (k+1)!^{-1} \mathscr H^{n-1} ( \mathbf S^{n-1} )
	(n+k+1)^{-1} t^{n+k+1} \| \Der^{k+1} P ( a ) \|$, we conclude
	\begin{equation*}
		\big | r^{-k-1-n} (T-S)_x \big ( \phi ( r^{-1} (x-a) ) \big )
		\big | \leq (M+N) \boldsymbol \nu_K^i ( \phi )
	\end{equation*}
	whenever $\phi \in\mathscr D_K ( \mathbf R^n, Y )$ and $0 < r \leq 1$.
	Consequently, we have
	\begin{equation*}
		\limsup_{r \to 0+} \big | r^{-k-1-n} (T-S)_x \big ( \theta (
		r^{-1} (x-a) ) \big ) \big | \leq (M+N) \boldsymbol \nu_K^i (
		\theta-\phi_j )
	\end{equation*}
	whenever $j$ is a positive integer, whence the principal conclusion
	follows.
	
	Noting \ref{remark:pt_diff} and~\ref{remark:norm_differentiability},
	the hypotheses of the postscript yield $0 \leq M < \infty$ and $s > 0$
	such that, whenever $0 < r \leq s$, we have
	\begin{equation*}
		\big | r^{-k-1-n} (T-S)_x \big ( \phi ( r^{-1} (x-a) ) \big )
		\big | \leq M \boldsymbol \nu_{\mathbf B(0,1)}^i ( \phi )
		\quad \text{for $\phi \in \mathscr D_{\mathbf B(0,1)} (
		\mathbf R^n, Y)$}.
	\end{equation*}
	As $\dim Y < \infty$ implies, by use of the Ascoli theorem
	(see~\cite[2.10.21]{MR41:1976}), that
	\begin{equation*}
		\mathscr D_{\mathbf B (0,1)} ( \mathbf R^n, Y ) \cap \big \{
		\phi \with \boldsymbol \nu_{\mathbf B(0,1)}^{i+1} ( \phi )
		\leq 1 \big \}
	\end{equation*}
	is $\boldsymbol \nu_{\mathbf B(0,1)}^i$ totally bounded, we readily
	combine the preceding estimate with the principal conclusion to obtain
	the postscript.
\end{proof}

\begin{remark}
	In view of \ref{remark:value_of_some_order}
	and~\ref{remark:bounded_of_some_order}, the case $n=1$, $k=-1$, $Y =
	\mathbf C$, $\Delta = \mathscr D ( \mathbf R, \mathbf C )$, and
	$\mathbf C$~linear $T$, was treated in~\cite[\S\,1.4 Satz]{MR0113134}.
\end{remark}

\begin{remark} \label{remark:dense_diff_crit}
	Combining \ref{remark:nu_pt_diff}, \ref{remark:norm_differentiability}
	and~\ref{lemma:dense_diff_crit}, we obtain the following proposition:
	\emph{Whenever $k$ is a nonnegative integer, $Y$ is a finite
	dimensional Banach space, $T \in \mathscr D' ( \mathbf R^n, Y )$, and
	$a \in \mathbf R^n$, the distribution $T$ is pointwise differentiable
	of order~$k$ at~$a$ if and only if, for some nonnegative integer~$i$,
	the distribution $T$~is $\boldsymbol \nu_{\mathbf B(0,1)}^i$ pointwise
	differentiable of order~$k$ at~$a$.}  The case $k=0$, $Y = \mathbf C$,
	and $\mathbf C$~linear $T$ of this characterisation also follows
	from~\cite[4.2 Théorème 1$'$]{MR0107167}.
\end{remark}

\begin{example} \label{example:denseness}
	The image $\Delta$ of $\mathscr D ( \mathbf R^n, \mathbf R ) \otimes
	Y$ in~$\mathscr D ( \mathbf R^n, Y )$ under the canonical monomorphism
	is sequentially dense in~$\mathscr D ( \mathbf R^n, Y )$
	by~\cite[3.1]{MR3528825}.
\end{example}
 
\begin{example} \label{example:nonmeasurable}
	There exists a separable Banach space $Y$ and $f : \mathbf R \to
	Y^\ast$, continuous with respect to the $Y$~topology on~$Y^\ast$, such
	that $\| f(x) - f (a) \| = 1$ whenever $a,x \in \mathbf R$ and $a \neq
	x$; in fact, noting~\ref{miniremark:lebesgue_duality} and \cite[2.1,
	2.6]{MR3626845}, we may take $Y \simeq L_1 ( \mathscr L^1, \mathbf
	R)$, and $f (x)$ to be associated to the characteristic function of
	the interval $\{ t x \with 0 \leq t \leq 1 \}$.  Evidently, such a
	function $f$ must be $\mathscr L^1$~nonmeasurable with respect to the
	norm topology on~$Y^\ast$ by~\cite[2.2.4]{MR41:1976}.
\end{example}

\begin{remark}
	The type of function considered originates
	from~\cite[p.\,265]{msb_v46_i2_p235}.
\end{remark}

\begin{lemma}
\label{lemma:distribution_function}
	Suppose $k$ is a nonnegative integer, $Y$ is a separable Banach space,
	$f$~is a $Y^\ast$~valued function, that is $\mathscr L^n$~measurable
	with respect to the $Y$~topology on~$Y^\ast$ and satisfies
	\begin{equation*}
		{\textstyle \int_K \| f \| \ud \mathscr L^n < \infty} \quad
		\text{whenever $K$ is a compact subset of~$\mathbf R^n$},
	\end{equation*}
	and $P : \mathbf R^n \to Y^\ast$ is a polynomial function of degree at
	most~$k$.
	
	Then, a distribution $S \in \mathscr D' ( \mathbf R^n, Y )$ may be
	defined by
	\begin{equation*}
		{\textstyle S (\phi) = \int \langle \phi, f \rangle \ud
		\mathscr L^n} \quad \text{for $\phi \in \mathscr D ( \mathbf
		R^n, Y )$}
	\end{equation*}
	and the following five statements hold:
	\begin{enumerate}
		\item \label{item:distribution_function:weak_diff} If $a \in
		\mathbf R^n$ satisfies
		\begin{gather*}
			\limsup_{r \to 0+} r^{-n-k} {\textstyle\int_{\mathbf
			B(a,r)}} \| f -P \| \ud \mathscr L^n < \infty, \\
			\lim_{r \to 0+} r^{-n-k} {\textstyle\int_{\mathbf B
			(a,r)} | \langle y, f(x) -P(x) \rangle | \ud
			\mathscr{L}^n \, x = 0} \quad \text{for $y \in Y$},
		\end{gather*}
		then $S$ is pointwise differentiable of order~$k$ at~$a$ and
		$\pt \Der^m S(a) = \Der^m P (a)$ for $m = 0, \ldots, k$.
		\item \label{item:distribution_function:weak_lebesgue} For
		$\mathscr L^n$ almost all~$a$, the distribution $S$ is
		pointwise differentiable of order~$0$ at~$a$ and $\pt \Der^0 S
		(a) = f(a)$.
		\item \label{item:distribution_function:norm} Whenever $a \in
		\mathbf R^n$ and $0 < r < \infty$, we have
		\begin{equation*}
			{\textstyle\int_{\mathbf B (a,r)}} \| f \| \ud
			\mathscr L^n = \sup \big \{ S ( \phi ) \with
			\text{$\phi \in \mathscr D_{\mathbf B(a,r)} ( \mathbf
			R^n, Y )$ and $\boldsymbol \nu_{\mathbf B(a,r)}^0 (
			\phi ) \leq 1$} \big \}.
		\end{equation*}
		\item \label{item:distribution_function:norm_diff} For $a \in
		\mathbf R^n$, the distribution $S$~is $\boldsymbol
		\nu_{\mathbf B (0,1)}^0$~pointwise differentiable of order~$k$
		at~$a$ with $\pt \Der^m S(a) = \Der^m P(a)$ for $m=0, \ldots,
		k$ if and only if
		\begin{equation*}
			\lim_{r \to 0+} r^{-n-k} {\textstyle\int_{\mathbf
			B(a,r)}} \| f - P \| \ud \mathscr L^n = 0.
		\end{equation*}
		\item \label{item:distribution_function:lebesgue} If $Y^\ast$
		is separable in its norm topology, then $S$~is $\boldsymbol
		\nu_{\mathbf B (0,1)}^0$~pointwise differentiable of order~$0$
		at $\mathscr L^n$ almost all~$a$.
	\end{enumerate}
\end{lemma}

\begin{proof}
	The legitimacy of the definition of~$S$ is noted
	in~\cite[4.1.1]{MR41:1976}.
	\eqref{item:distribution_function:weak_diff} is a consequence of
	\ref{lemma:dense_diff_crit} and \ref{example:denseness}.  Moreover,
	employing a countable dense subset of~$Y$ and \cite[2.8.18,
	2.9.9]{MR41:1976}, \eqref{item:distribution_function:weak_lebesgue}
	may be deduced from~\eqref{item:distribution_function:weak_diff}.
	Observing
	(see~\cite[4.6.1]{snulmenn:diploma_thesis}) that
	\begin{equation*}
		{\textstyle\int_{\mathbf B (a,r)}} \| f \| \ud \mathscr L^n =
		\sup \big \{ {\textstyle\int_{\mathbf B(a,r)} \langle \theta,
		f \rangle \ud \mathscr L^n} \with \theta \in \mathbf L_\infty
		( \mathscr L^n, Y ), ( \mathscr L^n )_{(\infty)} ( \theta )
		\leq 1 \big \}
	\end{equation*}
	for $a \in \mathbf R^n$ and $0 < r < \infty$,
	\eqref{item:distribution_function:norm} follows by suitably
	approximating such~$\theta$.  Finally,
	\eqref{item:distribution_function:norm}
	implies~\eqref{item:distribution_function:norm_diff} and,
	by~\cite[2.8.18, 2.9.9]{MR41:1976},
	\eqref{item:distribution_function:norm_diff}
	and~\ref{miniremark:setup} yield
	\eqref{item:distribution_function:lebesgue}.
\end{proof}

\begin{remark} \label{remark:separability}
	Under the hypothesis of~\eqref{item:distribution_function:lebesgue},
	the separability hypothesis on~$Y$ is redundant
	by~\cite[II.3.16]{MR0117523}.
\end{remark}

\begin{remark} \label{remark:need_dual_sep}
	In view of~\eqref{item:distribution_function:norm_diff},
	\ref{example:nonmeasurable} shows that the separability hypothesis
	on~$Y^\ast$ in~\eqref{item:distribution_function:lebesgue} may not be
	omitted even if $f$~is continuous with respect to the $Y$~topology.
	For such~$f$, the distribution~$S$ is $\boldsymbol \nu_{\mathbf
	B(0,1)}^0$ pointwise differentiable of order $(-1,1)$ at each $a \in
	\mathbf R^n$ by~\eqref{item:distribution_function:norm}, since $\| f
	\|$ is locally bounded by~\cite[II.3.21]{MR0117523}.
\end{remark}

\begin{remark}
	If $f$ is ``$w^\ast$-differentiable with
	$w^\ast$-dif\-fer\-en\-tial~$\psi$ at~$a$'' in the sense of
	\cite[3.4]{MR1800768} and $P(x)= f(a)+\langle x-a, \psi \rangle$ for
	$x \in \mathbf R^n$, then $S$ is pointwise differentiable of order~$1$
	at~$a$ and $\pt \Der^1 S(a) = \psi$
	by~\eqref{item:distribution_function:weak_diff}.
\end{remark}

\begin{example}
	Whenever $Y$ is a separable Hilbert space with $\dim Y = \infty$,
	there exists a distribution~$S$ associated to some~$f : \mathbf R \to
	Y$ as in~\ref{lemma:distribution_function} such that $S$ is
	$\boldsymbol \nu_{\mathbf B(0,1)}^0$ pointwise differentiable of
	order~$(-1,1)$ at~$0$ and $S$ is pointwise differentiable of order~$0$
	at~$0$ but, for every nonnegative integer~$i$, $S$ is not $\boldsymbol
	\nu_{\mathbf B(0,1)}^i$~pointwise differentiable of order~$0$ at~$0$;
	in fact, in view
	of~\ref{lemma:distribution_function}\,\eqref{item:distribution_function:weak_diff}\,\eqref{item:distribution_function:norm},
	choosing an orthonormal sequence~$y_j$ in~$Y$, we may take
	\begin{equation*}
		f(x) = y_j \quad \text{if $x \in \mathbf B(0,2^{1-j}) \without
		\mathbf U (0,2^{-j})$ for some positive integer $j$}
	\end{equation*}
	and $f(x) = 0$ if $x = 0$ or $x \in \mathbf R \without \mathbf
	B(0,1)$.
\end{example}

\begin{remark} \label{remark:dim_finite_not_omittable}
	The preceding example shows that, neither from the postscript
	of~\ref{lemma:dense_diff_crit} nor from~\ref{remark:dense_diff_crit},
	the hypothesis $\dim Y < \infty$ may be omitted.
\end{remark}

\begin{example} \label{example:distribution_representable_by_integration}
	Suppose $Y$ is a separable Banach space, $S \in \mathscr D' ( \mathbf
	R^n, Y)$ is representable by integration -- that is,
	see~\cite[4.1.5]{MR41:1976}, $\| S \|$ is a Radon measure and there
	exists a $Y^\ast \cap \{ \psi \with \| \psi \| = 1 \}$~valued
	function~$g$, that is $\| S \|$~measurable with respect to the
	$Y$~topology on~$Y^\ast$ and satisfies $S( \phi ) = \int \langle \phi,
	g \rangle \ud \| S \|$ for $\phi \in \mathscr D ( \mathbf R^n, Y)$ --,
	and $V = \{ (a,\mathbf B(a,r)) \with a \in \mathbf R^n, 0 < r < \infty
	\}$ is the standard $\mathscr L^n$~Vitali relation (see
	\cite[2.8.18]{MR41:1976}).  Then, for $\mathscr L^n$~almost all~$a$,
	the distribution~$S$ is pointwise -- in case $Y^\ast$ is separable,
	also $\boldsymbol \nu_{\mathbf B(0,1)}^0$~pointwise -- differentiable
	of order~$0$ at $a$ with
	\begin{equation*}
		\text{$\pt \Der^0 S(a) = \mathbf D ( \| S \|, \mathscr L^n, V,
		a ) g(a)$ if $a \in \dmn g$}, \quad \text{$\pt
		\Der^0 S (a) = 0$ else};
	\end{equation*}
	in fact, employing~\cite[2.9.2, 2.9.7]{MR41:1976} to express the
	absolutely continuous part~$\| S \|_{\mathscr L^n}$ of~$\| S \|$ with
	respect to~$\mathscr L^n$ via its $V$~derivate $\mathbf D ( \| S \|,
	\mathscr L^n, V, \cdot )$ and reducing to the case that $g$ is a Borel
	function with respect to the $Y$~topology and that $\dmn g = \mathbf
	R^n$ by \ref{miniremark:setup} and~\cite[2.3.6]{MR41:1976}, we
	use~\cite[2.4.10]{MR41:1976} to obtain
	\begin{equation*}
		S(\phi) = {\textstyle\int} \langle \phi(x), g (x) \rangle
		\mathbf D ( \| S \|, \mathscr L^n, V, x ) \ud \mathscr L^n \,
		x + {\textstyle\int} \langle \phi, g \rangle \ud ( \| S \| -
		\| S \|_{\mathscr L^n} )
	\end{equation*}
	whenever $\phi \in \mathscr D ( \mathbf R^n, Y )$, whence we infer the
	conclusion
	by~\ref{lemma:distribution_function}\,\eqref{item:distribution_function:weak_lebesgue}\,\eqref{item:distribution_function:lebesgue}
	as we are assured that $\mathbf D ( \| S \| - \| S \|_{\mathscr L^n},
	\mathscr L^n, V, x ) = 0$ for $\mathscr L^n$~almost all~$x$ by
	\cite[2.9.10]{MR41:1976}.
\end{example}

\section{Poincaré inequality}

The main purpose of this section is to establish the indicated Poincaré
inequality
in~\ref{miniremark:hahn_banach}--\ref{remark:allards_strong_constancy}.
Furthermore, we include its applications to the study of the relation of
pointwise differentiability and distributional derivatives
in~\ref{lemma:polynomial_fcts}--\ref{remark:point_to_point}.

\begin{miniremark} \label{miniremark:hahn_banach}
	Suppose $i$ is a nonnegative integer, $K$ is a compact subset of
	$\mathbf R^n$, $Y$ is a separable Banach space, $T \in \mathscr D' (
	\mathbf R^n, Y )$, and $0 \leq \kappa < \infty$.  Then, we will verify
	(for use in~\ref{remark:allards_strong_constancy} only) the
	equivalence of the following two conditions.
	\begin{enumerate}
		\item \label{item:hahn_banach:condition} For $\phi \in
		\mathscr D_K ( \mathbf R^n, Y )$, we have $| T ( \phi ) | \leq
		\kappa \sup \im \| \Der^i \phi \|$.
		\item \label{item:hahn_banach:extension} There exists $S \in
		\mathscr D' \big ( \mathbf R^n, \bigodot^i ( \mathbf R^n, Y )
		\big )$ satisfying $\spt S \subset K$, $\| S \| ( \mathbf R^n
		) \leq \kappa$, and
		\begin{equation*}
			T ( \phi ) = S ( \Der^i \phi ) \quad \text{for $\phi
			\in \mathscr D_K ( \mathbf R^n, Y)$}.
		\end{equation*}
	\end{enumerate}
	In fact, proceeding as in~\cite[4.1.12]{MR41:1976},
	if~\eqref{item:hahn_banach:condition} holds, $V = \mathscr D \big (
	\mathbf R^n, \bigodot^i ( \mathbf R^n,Y ) \big)$, $\sigma$ is the
	seminorm on $V$ defined by $\sigma ( v ) = \kappa \sup \{ \| v (x) \|
	\with x \in K \}$ for $v \in V$, and $Q : \mathscr D_K ( \mathbf R^n,
	Y) \to V$ is the linear monomorphism induced by $\Der^i$, then we have
	$T \circ Q^{-1} \leq \sigma | \im Q$ and
	\eqref{item:hahn_banach:extension} follows by extending $T \circ
	Q^{-1}$ to a linear map $S : V \to \mathbf R$ with $S \leq \sigma$ by
	means of the Hahn-Banach theorem \cite[2.4.12]{MR41:1976}.
\end{miniremark}

\begin{miniremark} \label{miniremark:lip}
	Suppose $i$ is a nonnegative integer, $a \in \mathbf R^n$, $0 < r <
	\infty$, and $Y$ is a Banach space.  Then, there holds $\sup \im \|
	\Der^m \phi \| \leq r^{i-m} \sup \im \| \Der^i \phi \|$ for $m = 0,
	\ldots, i$ and $\phi \in \mathscr D_{\mathbf B(a,r)} ( \mathbf R^n, Y
	)$ by~\cite[2.2.7, 3.1.1, 3.1.11]{MR41:1976}; in particular, we have
	\begin{equation*}
		\boldsymbol \nu_{\mathbf B(0,1)}^i ( \phi ) = \sup \im \|
		\Der^i \phi \| \quad \text{for $\phi \in \mathscr D_{\mathbf
		B(0,1)} ( \mathbf R^n, Y)$}.
	\end{equation*}
\end{miniremark}

\begin{theorem} \label{thm:poincare_inequality}
	Suppose $i$ is a nonnegative integer, $k$ is a positive integer, $Y$
	is a Banach space, $T \in \mathscr D' ( \mathbf R^n, Y )$, $0 < r <
	\infty$, $\Phi \in \mathscr D_{\mathbf B (0,r)} ( \mathbf R^n, \mathbf
	R )$,
	\begin{equation*}
		\Phi \ast Q = Q \quad
	\end{equation*}
	whenever $Q : \mathbf R^n \to \mathbf R$ is a polynomial function of
	degree at most~$k-1$, $0 \leq \kappa < \infty$, $C$~is a compact
	convex subset of~$\mathbf R^n$, $K = \bigcup \{ \mathbf B (c,kr) \with
	c \in C \}$,
	\begin{equation*}
		| (\Der^o T) ( \phi ) | \leq \kappa \sup \im \| \Der^i
		\phi \| \quad \text{for $\phi \in \mathscr D_K ( \mathbf
		R^n, Y)$ and $o \in \boldsymbol \Xi (n,k)$},
	\end{equation*}
	$a \in C$, and $P : \mathbf R^n \to Y^\ast$ is the polynomial function
	of degree at most~$k-1$ satisfying
	\begin{equation*}
		\langle y, \Der^\xi P ( a ) \rangle = (\Der^\xi T)_b ( \Phi
		(b-a) y)
	\end{equation*}
	whenever $m = 0, \ldots, k-1$, $\xi \in \boldsymbol \Xi (n,m)$, and $y
	\in Y$.

	Then, for $m = 0, \ldots, k-1$ and $\xi \in \boldsymbol \Xi (n,m)$,
	there holds
	\begin{equation*}
		\big | ( \Der^\xi T ) ( \theta ) - {\textstyle\int} \langle
		\theta, \Der^\xi P \rangle \ud \mathscr L^n \big | \leq
		\Gamma_m
		\kappa \sup \im \| \Der^i \theta \| \quad \text{for $\theta
		\in \mathscr D_C ( \mathbf R^n, Y )$},
	\end{equation*}
	where $\Gamma_m = (n \boldsymbol \alpha (n) \sup \im \| \Der^i \Phi
	\|)^{k-m} \prod_{\mu=1}^{k-m} (2\mu r+\diam C)^{1+n+i}$.
\end{theorem}

\begin{proof}
	We assume $a = 0$ and abbreviate $\delta = \diam C$ and $\lambda =
	\sup \im \| \Der^i \Phi \|$.  Replacing $T$ and $k$ by $\Der^\xi T$
	and $k-m$, it suffices to show the assertion
	\begin{equation*}
		\big | T ( \theta ) - {\textstyle\int} \langle \theta, P
		\rangle \ud \mathscr L^n \big | \leq \Gamma_0 \kappa \sup \im
		\| \Der^i \theta \| \quad \text{for $\theta \in \mathscr D_C (
		\mathbf R^n, Y )$}.
	\end{equation*}

	Firstly, the special case $k=1$ thereof will be proven.  Suppose
	$\theta \in \mathscr D_C ( \mathbf R^n, Y )$ with $\sup \im \| \Der^i
	\theta \| \leq 1$.  Using the coordinate functions $X_j : \mathbf R^n
	\to \mathbf R$ given by $X_j (x) = x_j$ whenever $x=(x_1, \ldots, x_n)
	\in \mathbf R^n$, we define $\phi_j : \mathbf R^n \to Y$ by
	\begin{equation*}
		\phi_j (x) = {\textstyle\iint_0^1} \Phi ( b ) X_j (b)
		\theta (x-tb) \ud \mathscr L^1 \, t \ud \mathscr L^n \, b
	\end{equation*}
	whenever $x \in \mathbf R^n$ and $j = 1, \ldots, n$.  We note $\phi_j
	\in \mathscr D_K ( \mathbf R^n, Y )$ with
	\begin{equation*}
		\Der^m \phi_j (x) = {\textstyle\iint_0^1} \Phi ( b ) X_j (b)
		\Der^m \theta (x-tb) \ud \mathscr L^1 \, t \ud \mathscr L^n \,
		b
	\end{equation*}
	for $m = 0,1,2, \ldots$ and $x \in \mathbf R^n$ and infer
	\begin{equation*}
		\sup \im \| \Der^i \phi_j \| \leq r (\mathscr L^n)_{(1)} (
		\Phi ) \leq \boldsymbol \alpha (n) r^{1+n+i} \lambda \quad
		\text{for $j = 1, \ldots, n$}
	\end{equation*}
	by~\ref{miniremark:lip} and also $\theta - \Phi \ast \theta =
	\sum_{j=1}^n \Der_j \phi_j$ since
	\begin{equation*}
		\theta (x) - ( \Phi \ast \theta ) (x) = {\textstyle\int} \Phi
		(b) ( \theta (x) - \theta (x-b) ) \ud \mathscr L^n \, b =
		\sum_{j=1}^n \Der_j \phi_j (x) \quad \text{for $x \in \mathbf
		R^n$}.
	\end{equation*}
	Consequently, we obtain $|T ( \theta- \Phi \ast \theta ) | \leq n
	\boldsymbol \alpha (n) r^{1+n+i} \kappa \lambda$.  In view
	of~\cite[4.1.2]{MR41:1976}, we may define $f \in \mathscr E ( \mathbf
	R^n, Y^\ast )$ by requiring
	\begin{equation*}
		\langle y, f(x) \rangle = T_b ( \Phi ( b-x ) y ) \quad
		\text{for $x \in \mathbf R^n$ and $y \in Y$}.
	\end{equation*}
	Since $\langle y, \Der_j f (x) \rangle = ( \Der_j T )_b ( \Phi (b-x) y
	)$ for $y \in Y$, we infer
	\begin{equation*}
		\| \Der_j f (x) \| \leq \kappa \lambda \quad \text{for $x \in
		C$ and $j = 1, \ldots, n$}.
	\end{equation*}
	Recalling $T ( \Phi \ast \theta ) = \int \langle \theta, f \rangle \ud
	\mathscr L^n$ from~\cite[4.1.2]{MR41:1976} and noting $\sup \im |
	\theta | \leq \delta^i$, where $0^0=1$, by~\ref{miniremark:lip}, the
	preceding estimate implies
	\begin{equation*}
		\big | T ( \Phi \ast \theta ) - {\textstyle\int} \langle
		\theta (x), f(0) \rangle \ud \mathscr L^n \, x \big | \leq n
		\boldsymbol \alpha (n) \delta^{1+n+i} \kappa \lambda
	\end{equation*}
	and the conclusion follows in the present case.

	Proceeding inductively, we now establish that the validity of the
	assertion for some~$k$ implies its validity for $k+1$.  For this
	purpose, we observe that $D = \bigcup \{ \mathbf B (c,kr) \with c \in
	C \}$ is a convex set and define $S \in \mathscr D' ( \mathbf R^n, Y
	)$ by
	\begin{equation*}
		S ( \theta ) = T ( \theta ) - {\textstyle\int} \big \langle
		\theta (x), \langle x^k/k!, \Der^k P(0) \rangle \big
		\rangle \ud \mathscr L^n \, x \quad \text{for $\theta \in
		\mathscr D ( \mathbf R^n, Y )$}.
	\end{equation*}
	For $\theta \in \mathscr D_D ( \mathbf R^n, Y )$ and $\xi \in
	\boldsymbol \Xi (n,k)$, noting
	\begin{equation*}
		( \Der^\xi S ) ( \theta ) = ( \Der^\xi T ) ( \theta ) -
		{\textstyle\int} \langle \theta (x), \Der^\xi P (0) \rangle
		\ud \mathscr L^n \, x,
	\end{equation*}
	we apply the special case with $T$ and $C$ replaced by $\Der^\xi T$
	and $D$ to conclude
	\begin{equation*}
		| ( \Der^\xi S ) ( \theta ) | \leq n \boldsymbol \alpha (n)
		\lambda ( 2(k+1)r + \delta )^{1+n+i} \kappa \sup \im \| \Der^i
		\theta \|.
	\end{equation*}
	Whenever $y \in Y$, $m = 0, \ldots, k-1$, and $\xi \in \boldsymbol \Xi
	(n,m)$, we consider the polynomial function $Q : \mathbf R^n \to
	\mathbf R$ defined by $Q(b) = \langle y, \langle (-b)^k/k!, \Der^k P
	(0) \rangle \rangle$ for $b \in \mathbf R^n$ and compute
	\begin{equation*}
		{\textstyle\int} \Der^\xi \Phi (b) \big \langle y, \langle
		b^k/k!, \Der^k P(0) \rangle \big \rangle \ud \mathscr L^n \, b
		= \big ((\Der^\xi \Phi ) \ast Q \big )(0) = \Der^\xi Q (0) = 0
	\end{equation*}
	to infer
	\begin{equation*}
		( \Der^\xi S )_b ( \Phi (b) y ) = ( \Der^\xi T )_b ( \Phi
		(b) y ).
	\end{equation*}
	Finally,  in view of the preceding estimate, the conclusion for $k+1$
	follows from the conclusion for $k$ with $T$ and $\kappa$ replaced by
	$S$ and $n \boldsymbol \alpha (n) \lambda ( 2(k+1)r+\delta )^{1+n+i}
	\kappa$.
\end{proof}

\begin{remark}
	The estimate of $T ( \theta - \Phi \ast \theta )$ is adapted from the
	use of the smoothing homotopy formulae for currents
	in~\cite[4.1.18]{MR41:1976}.
\end{remark}

\begin{remark} \label{remark:cz_special_smoothing}
	From~\cite[Lemma~2.6]{MR0136849} or \cite[3.5.6]{MR1014685},
	we recall the following proposition.  \emph{If $k$ is a positive
	integer, then there exists $\Phi \in \mathscr D_{\mathbf B (0,1)} (
	\mathbf R^n, \mathbf R )$ satisfying $\Phi_r \ast Q = Q$ whenever $0 <
	r < \infty$ and $Q : \mathbf R^n \to \mathbf R$ is a polynomial
	function of degree at most~$k-1$, where $\Phi_r (x) = r^{-n} \Phi
	(r^{-1} x)$ for $x \in \mathbf R^n$.}
\end{remark}

\begin{remark}
	If, for some positive integer $k$, a function $\Phi \in \mathscr D (
	\mathbf R^n, \mathbf R )$ is such that $\Phi \ast Q = Q$ for every
	polynomial function $Q : \mathbf R^n \to \mathbf R$ of degree at
	most~$k-1$, then the functions $\phi_o \in \mathscr D ( \mathbf R^n,
	\mathbf R)$, defined by $\phi_o (x) = (o!)^{-1} \Der^o \Phi(-x)$ for
	$x \in \mathbf R^n$, $m = 0, \ldots, k-1$, and $o \in \boldsymbol \Xi
	(n,m)$, satisfy the conditions
	\begin{equation*}
		\text{${\textstyle\int} \phi_o (x) x^\xi \ud \mathscr L^n \, x
		= 1$ if $o = \xi$}, \quad \text{${\textstyle\int} \phi_o (x)
		x^\xi \ud \mathscr L^n \, x = 0$ if $o \neq \xi$}
	\end{equation*}
	whenever $m = 0, \ldots, k-1$ and $\xi \in \boldsymbol \Xi (n,m)$,
	where $x^\xi = \prod_{j=1}^n (x_j)^{\xi_j}$ and $0^0 = 1$.  Such a
	family of functions was employed in~\cite[p.~297]{MR0225159-english}
	to construct, for a purpose very similar to ours, a ``projection
	operator'' of $\mathscr D'( \mathbf R^n, \mathbf R)$ onto the subspace
	of distributions corresponding to a polynomial function of degree at
	most~$k-1$.  We also note that, for $o \in \boldsymbol \Xi (n,0)$, the
	condition on $\phi_o$ is equivalent to the invariance property in
	question.
\end{remark}

\begin{remark} \label{remark:allards_strong_constancy}
	It is illustrative to compare the preceding theorem (for $i = 1 = k$
	and $Y = \mathbf R$) to the strong constancy lemma
	of~\cite[\S\,1]{MR840267} which, in view
	of~\ref{miniremark:hahn_banach}, may be restated as follows.
	\emph{Whenever $n$ is a positive integer, $0 < \lambda < 1$, and
	$\epsilon > 0$, there exists $0 \leq \Gamma < \infty$ with the
	following property:  If $a \in \mathbf R^n$, $0 < r < \infty$, $T \in
	\mathscr D_{\mathbf B(a,r)}' ( \mathbf R^n, \mathbf R)$, $T ( \phi)
	\geq 0$ for $0 \leq \phi \in \mathscr D ( \mathbf R^n, \mathbf R)$, $0
	\leq \kappa < \infty$, and
	\begin{equation*}
		| ( \Der_j T ) ( \phi ) | \leq \kappa \sup \im | \Der \phi |
		\quad \text{for $\phi \in \mathscr D_{\mathbf B (a,r)} (
		\mathbf R^n, \mathbf R )$ and $j = 1, \ldots, n$},
	\end{equation*}
	then there exists $0 \leq c < \infty$ such that
	\begin{equation*}
		\big | T(\theta)-c{\textstyle\int} \theta \ud \mathscr L^n
		\big | \leq \big ( \epsilon \| T \| \, \mathbf B (a,r) +
		\Gamma \kappa \big ) \sup \im | \theta |
	\end{equation*}
	for $\theta \in \mathscr D_{\mathbf B (a,\lambda r)} ( \mathbf R^n,
	\mathbf R )$.} In our case, $T$ need not correspond to a monotone
	Daniell integral and the summand $\| T \| \, \mathbf B(a,r)$ does not
	occur; accordingly, a stronger norm is employed for~$\theta$.
\end{remark}

\begin{lemma} \label{lemma:polynomial_fcts}
	Suppose $k$ is a nonnegative integer, $n$ is a positive integer, and
	$\nu$~is a norm on $\mathscr D_{\mathbf B(0,1)} ( \mathbf R^n, \mathbf
	R )$, $v \in \mathbf S^{n-1}$, $K = \mathbf B (v,1) \cap \mathbf B
	(0,1)$, and $P : \mathbf R^n \to \mathbf R$ is a polynomial function
	of degree at most~$k$.

	Then, for some $0 \leq \Gamma < \infty$ determined by $k$, $n$, and
	$\nu$, there holds
	\begin{multline*}
		\sup \{ \| \Der^m P (x) \| \with x \in \mathbf B (0,1), m = 0,
		\ldots, k \} \\
		\leq \Gamma \sup \big \{ {\textstyle \int} \phi P \ud \mathscr
		L^n \with \phi \in \mathscr D_K ( \mathbf R^n, \mathbf R ),
		\nu ( \phi ) \leq 1 \big \}.
	\end{multline*}
\end{lemma}

\begin{proof}
	Using rotations, it is sufficient to consider a fixed vector~$v$.
	Then, we notice (see~\cite[2.6.5]{MR41:1976}) that both sides
	represent norms on the finite dimensional space of real valued
	polynomial functions on~$\mathbf R^n$ of degree at most~$k$.
\end{proof}

\begin{remark}
	The particular form of the set $K$ will be employed in the proof
	of~\ref{thm:cka}.  For the present section, $K = \mathbf B(0,1)$ would
	be sufficient.
\end{remark}

\begin{remark} \label{remark:polynomial_fcts}
	Whenever $i$ is a nonnegative integer, we may take $\nu = \boldsymbol
	\nu_{\mathbf B(0,1)}^i$ and employ~\ref{miniremark:lip} to infer, for
	$a \in \mathbf R^n$, $0 < s < \infty$, and $m = 0, \ldots, k$, that
	\begin{align*}
		& s^m \| \Der^m P (a) \| \\
		& \qquad \leq \Gamma s^{-n-i} \sup \big \{ {\textstyle \int}
		\phi P \ud \mathscr L^n \with \phi \in \mathscr D_{\mathbf
		B(a,s)} ( \mathbf R^n, \mathbf R ), \sup \im \| \Der^i \phi \|
		\leq 1 \big \}.
	\end{align*}
\end{remark}

\begin{theorem} \label{thm:pt_diff_derivatives}
	Suppose $i$, $k$, and $l$ are integers, $i \geq 0$, $k \geq 1$, $0 <
	\alpha \leq 1$, $l+\alpha \geq 0$, $Y$ is a Banach space, $T \in
	\mathscr D' ( \mathbf R^n, Y )$, and $a \in \mathbf R^n$.

	Then, the following four statements hold.
	\begin{enumerate}
		\item \label{item:pt_diff_derivatives:integral} If $l \geq 0$
		and, for $o \in \boldsymbol \Xi (n,k)$, the distribution
		$\Der^o T$ is $\boldsymbol \nu_{\mathbf B (0,1)}^i$~pointwise
		differentiable of order~$l$ at~$a$, then, for $m = 0, \ldots,
		k-1$ and $\xi \in \boldsymbol \Xi (n,m)$, the distribution
		$\Der^\xi T$ is $\boldsymbol \nu_{\mathbf B
		(0,1)}^i$~pointwise differentiable of order~$k+l-m$ at~$a$.
		\item \label{item:pt_diff_derivatives:hoelder} If, for $o \in
		\boldsymbol \Xi (n,k)$, the distribution $\Der^o T$ is
		$\boldsymbol \nu_{\mathbf B (0,1)}^i$~pointwise differentiable
		of order~$(l,\alpha)$ at~$a$, then, for $m = 0, \ldots, k-1$
		and $\xi \in \boldsymbol \Xi (n,m)$, the distribution
		$\Der^\xi T$ is $\boldsymbol \nu_{\mathbf B
		(0,1)}^i$~pointwise differentiable of order~$(k+l-m,\alpha)$
		at~$a$.
		\item \label{item:pt_diff_derivatives:top_integral} If $l \geq
		0$ and, for $o \in \boldsymbol \Xi (n,k)$, the distribution
		$\Der^o T$ is pointwise differentiable of order~$l$ at~$a$,
		then, for $m = 0, \ldots, k-1$ and $\xi \in \boldsymbol \Xi
		(n,m)$, the distribution $\Der^\xi T$ is pointwise
		differentiable of order~$k+l-m$ at~$a$.
		\item \label{item:pt_diff_derivatives:top_hoelder} If, for $o
		\in \boldsymbol \Xi (n,k)$, the distribution $\Der^o T$ is
		pointwise differentiable of order~$(l,\alpha)$ at~$a$, then,
		for $m = 0, \ldots, k-1$ and $\xi \in \boldsymbol \Xi (n,m)$,
		the distribution $\Der^\xi T$ is pointwise differentiable of
		order~$(k+l-m,\alpha)$ at~$a$.
	\end{enumerate}
	Moreover, under the hypotheses of any of these statements, there
	holds (see~\ref{remark:distributional_derivative})
	\begin{equation*}
		e^{o-\xi} \mathop{\lrcorner} \pt \Der^{k-m+\mu} (\Der^\xi T)
		(a) = \pt \Der^\mu (\Der^o T) (a)
	\end{equation*}
	whenever $m = 0, \ldots, k-1$, $\xi \in \boldsymbol \Xi (n,m)$, $\mu =
	0, \ldots, l$, $o \in \boldsymbol \Xi (n,k)$, and $\xi \leq o$.
\end{theorem}

\begin{proof}
	We will reduce the problem by adding the condition $\pt \Der^\mu (
	\Der^o T) ( a) = 0$ for $\mu = 0, \ldots, l$ and $o \in \boldsymbol
	\Xi (n,k)$ to the hypotheses of the four statements.  Indeed, noting
	\ref{remark:pt_diff}, the hypotheses of any of the four statements
	allow to apply \ref{remark:distributional_derivative} with $T$, $k$,
	$m$, $\xi$, and $\mu$ replaced by $\Der^o T$, $\mu$, $\mu$, $\pi$, and
	$0$ to conclude that $\Der^{o+\pi} T = \Der^\pi \Der^o T$ is pointwise
	differentiable of order~$0$ at~$a$ with
	\begin{equation*}
		\pt \Der^0 ( \Der^{o+\pi} T ) (a) = \langle e^\pi, \pt
		\Der^\mu ( \Der^o T ) (a) \rangle
	\end{equation*}
	for $\mu = 0, \ldots, l$, $\pi \in \boldsymbol \Xi (n,\mu)$, and $o
	\in \boldsymbol \Xi (n,k)$.  Defining $\psi_\mu \in \bigodot^{k+\mu} (
	\mathbf R^n, Y^\ast )$ by $\langle e^\rho, \psi_\mu \rangle = \pt
	\Der^0 ( \Der^\rho T ) (a)$ for $\mu = 0, \ldots, l$ and $\rho \in
	\boldsymbol \Xi (n,k+\mu)$ and $Q : \mathbf R^n \to Y^\ast$ by $Q(x) =
	\sum_{\mu=0}^l \langle (x-a)^{k+\mu}/(k+\mu)!, \psi_\mu \rangle$ for
	$x \in \mathbf R^n$, we then infer
	\begin{equation*}
		\Der^\mu \Der^o Q (a) = e^o \mathop{\lrcorner} \psi_\mu = \pt
		\Der^\mu ( \Der^o T ) (a) \quad \text{for $\mu = 0, \ldots, l$
		and $o \in \boldsymbol \Xi (n,k)$},
	\end{equation*}
	whence the indicated reduction follows by replacing $T(\theta)$ by
	$T(\theta)-\int \langle \theta, Q \rangle \ud \mathscr L^n$ for
	$\theta \in \mathscr D ( \mathbf R^n, Y )$.

	Next, we establish \eqref{item:pt_diff_derivatives:integral}
	and~\eqref{item:pt_diff_derivatives:hoelder} and the validity of the
	formula in the postscript under the hypotheses of any of these two
	statements.  For this purpose, we define $\gamma = l$ in case
	of~\eqref{item:pt_diff_derivatives:integral} and $\gamma = l+\alpha$
	in case of~\eqref{item:pt_diff_derivatives:hoelder}, hence $\gamma
	\geq 0$.  We choose $\Phi$ and $\Phi_r$ as
	in~\ref{remark:cz_special_smoothing}.  For $0 < r < \infty$, we also
	define polynomial functions $P_r : \mathbf R^n \to Y^\ast$ of degree
	at most~$k-1$ characterised by
	\begin{equation*}
		\langle y, \Der^\xi P_r(a) \rangle = ( \Der^\xi T)_b ( \Phi_r
		(b-a) y )
	\end{equation*}
	whenever $m = 0, \ldots, k-1$, $\xi \in \boldsymbol \Xi (n,m)$, and $y
	\in Y$, abbreviate $C(r) = \mathbf B(a,r)$ and $K(r) = \mathbf B
	(a,(k+1)r)$, and let $\kappa (r)$ denote the supremum of the set of
	all numbers
	\begin{equation*}
		s^{-n-\gamma-i} ( \Der^o T ) ( \phi )
	\end{equation*}
	corresponding to $0 < s \leq r$, $o \in \boldsymbol \Xi (n,k)$, and
	$\phi \in \mathscr D_{K(s)} ( \mathbf R^n, Y )$ satisfying $\sup \im
	\| \Der^i \phi \| \leq 1$.  By~\ref{miniremark:lip}, the hypotheses
	of~\eqref{item:pt_diff_derivatives:integral}
	and~\eqref{item:pt_diff_derivatives:hoelder} yield
	\begin{equation*}
		\text{$\lim_{r \to 0+} \kappa (r)=0$ in case
		of~\eqref{item:pt_diff_derivatives:integral}}, \quad
		\text{$\limsup_{r \to 0+} \kappa (r) < \infty$ in case
		of~\eqref{item:pt_diff_derivatives:hoelder}};
	\end{equation*}
	in particular, there exists $0 < \delta < \infty$ with $\kappa (
	\delta ) < \infty$.  For $0 < r \leq \delta$, applying
	\ref{thm:poincare_inequality} with $\Phi$, $\kappa$, and $C$ replaced
	by $\Phi_r$, $r^{n+\gamma+i} \kappa(r)$, and $C(r)$, we obtain
	\begin{align*}
		& \big | ( \Der^\xi T ) ( \theta ) - {\textstyle\int} \langle
		\theta, \Der^\xi P_r \rangle \ud \mathscr L^n \big | \leq
		\Delta_1 r^{k-m+n+\gamma+i} \kappa (r), \\
		& \qquad \text{where $\Delta_1 = \big ( ( n \boldsymbol \alpha
		(n) \sup \im \| \Der^i \Phi \| )^k+1 \big ) \big ( (k+1)!2^k
		\big )^{1+n+i}$},
	\end{align*}
	for $m = 0, \ldots, k-1$, $\xi \in \boldsymbol \Xi (n,m)$, and
	$\theta \in \mathscr D_{C(r)} ( \mathbf R^n, Y )$ with $\sup \im \|
	\Der^i \theta \| \leq 1$.  Moreover, whenever
	$0 < r/2 \leq s \leq r \leq \delta$, noting
	\begin{equation*}
		\big | {\textstyle \int} \langle \theta, P_r-P_s \rangle \ud
		\mathscr L^n \big | \leq 2 \Delta_1 r^{k+n+\gamma+i} \kappa
		(r) \sup \im \| \Der^i \theta \| \quad \text{for $\theta \in
		\mathscr D_{C(s)} ( \mathbf R^n, Y )$},
	\end{equation*}
	\ref{remark:polynomial_fcts} implies
	\begin{equation*}
		\| \Der^m ( P_r-P_s) (a) \| \leq \Delta_2 r^{k-m+\gamma}
		\kappa (r) \quad \text{for $m=0, \ldots, k-1$},
	\end{equation*}
	where $0 \leq \Delta_2 < \infty$ is determined by $i$, $k$, $n$, and
	$\Phi$.  Therefore, as $k-m+\gamma \geq 1$, we may define $P : \mathbf
	R^n \to Y^\ast$ by $P(x) = \sum_{m=0}^{k-1} \langle (x-a)^m/m!,
	\lim_{r \to 0+} \Der^m P_r (a) \rangle$ for $x \in \mathbf R^n$ and
	estimate
	\begin{equation*}
		\| \Der^m (P_r-P) (a) \| \leq 2 \Delta_2 r^{k-m+\gamma} \kappa
		(r) \quad \text{for $m = 0, \ldots, k-1$ and $0 < r \leq
		\delta$},
	\end{equation*}
	both using the geometric series.  Employing \ref{miniremark:lip}
	for~$\theta$ and Taylor's formula (see~\cite[1.10.4,
	3.1.11]{MR41:1976}) to bound $\| \Der^\xi (P_r-P) (x)\|$ for $x \in
	C(r)$, we conclude
	\begin{gather*}
		\big | {\textstyle\int} \langle \theta, \Der^\xi (P_r-P)
		\rangle \ud \mathscr L^n \big | \leq 6 \boldsymbol \alpha (n)
		\Delta_2 r^{k-m+n+\gamma+i} \kappa (r) , \\
		\big | ( \Der^\xi T ) (\theta) - {\textstyle\int} \langle
		\theta, \Der^\xi P \rangle \ud \mathscr L^n \big | \leq (
		\Delta_1 + 6 \boldsymbol \alpha (n) \Delta_2 )
		r^{k-m+n+\gamma+i} \kappa (r)
	\end{gather*}
	for $m = 0, \ldots, k-1$, $\xi \in \boldsymbol \Xi (n,m)$, and
	$\theta \in \mathscr D_{C(r)} ( \mathbf R^n, Y )$ with $\sup \im \|
	\Der^i \theta \| \leq 1$, whence we infer
	\eqref{item:pt_diff_derivatives:integral}
	and~\eqref{item:pt_diff_derivatives:hoelder} and the corresponding
	part of the postscript.

	Combining \eqref{item:pt_diff_derivatives:hoelder}
	and~\ref{remark:norm_differentiability}, we
	obtain~\eqref{item:pt_diff_derivatives:top_hoelder} and its part of
	the postscript.

	To establish~\eqref{item:pt_diff_derivatives:top_integral}, we firstly
	notice that the case $\dim Y < \infty$
	of~\eqref{item:pt_diff_derivatives:top_integral}, including its
	postscript, follows from \eqref{item:pt_diff_derivatives:integral} and
	\ref{remark:dense_diff_crit}.  To treat the general case
	of~\eqref{item:pt_diff_derivatives:top_integral}, we define $T^y \in
	\mathscr D' ( \mathbf R^n, \mathbf R )$ by $T^y ( \zeta ) = T_x (
	\zeta(x)y)$ for $y \in Y$ and $\zeta \in \mathscr D ( \mathbf R^n,
	\mathbf R )$.  From the case $\dim Y < \infty$
	of~\eqref{item:pt_diff_derivatives:top_integral} and its postscript,
	we conclude that $T^y$ is pointwise differentiable of order~$k+l$
	at~$a$ and $\pt \Der^{k+\mu} T^y (a) = 0$ for $\mu = 0, \ldots, l$.
	By~\eqref{item:pt_diff_derivatives:top_hoelder}, $T$ is pointwise
	differentiable of order $(k+l-1,1)$ at~$a$.  Defining $P :
	\mathbf R^n \to Y^\ast$ by
	\begin{equation*}
		P(x) = \sum_{m=0}^{k-1} \langle (x-a)^m/m!, \pt \Der^m T (a)
		\rangle \quad \text{for $x \in \mathbf R^n$},
	\end{equation*}
	we readily verify
	\begin{equation*}
		\langle \eta, \pt \Der^\iota T^y (a) \rangle = \big \langle y,
		\langle \eta, \Der^\iota P(a) \rangle \big \rangle
	\end{equation*}
	whenever $\iota = 0, \ldots, k+l$, $\eta \in \bigodot_\iota \mathbf
	R^n$, and $y \in Y$.  In combination with~\ref{example:denseness}, we
	then apply~\ref{lemma:dense_diff_crit} with $k$ replaced by $k+l-1$ to
	infer that $T$ is pointwise differentiable of order~$k+l$ at~$a$ and
	that $P$ corresponds to the $k+l$~jet of~$T$ at~$a$.  Finally,
	\ref{remark:distributional_derivative}~with $k$ and $\mu$ replaced by
	$k+l$ and $k-m+\mu$ yields the remaining cases
	of~\eqref{item:pt_diff_derivatives:top_integral} and the postscript.
\end{proof}

\begin{remark} \label{remark:Estrada}
	Instead of treating the preceding theorem as corollary to the
	principal theorem of this section, the Poincaré inequality in
	\ref{thm:poincare_inequality}, one may also approach it directly
	by means of homothetic deformations: Inductively reducing the
	principal statements to the case $k=1$ and $m=0$, assuming $\pt
	\Der^\mu ( \Der_j T ) (0) = 0$ for $\mu = 0, \ldots, l$ and $j = 1,
	\ldots, n$, and defining $T^r$ and $S_j^r$ in $\mathscr D' ( \mathbf
	R^n, Y )$ by
	\begin{equation*}
		T^r ( \phi ) = T_x \big ( r^{-n} \phi ( r^{-1} (x-a)) \big ),
		\quad S_j^r ( \phi ) = ( \Der_j T )_x \big ( r^{-n} \phi (
		r^{-1} (x-a) ) \big )
	\end{equation*}
	one verifies, by means of differentiation with respect to $r$,
	that, for $0 < s \leq r < \infty$,
	\begin{equation*}
		(T^r-T^s) ( \phi ) = \sum_{j=1}^n {\textstyle\int_s^r} S_j^r (
		\phi X_j ) \ud \mathscr L^1 \, r \quad \text{whenever
		$\phi \in \mathscr D ( \mathbf R^n, Y )$}
	\end{equation*}
	where $X_j : \mathbf R^n \to \mathbf R$ satisfy $X_j (x) = x_j$
	for $x=(x_1, \ldots, x_n ) \in \mathbf R^n$, whence one may deduce,
	employing the principle of uniform boundedness as in
	\ref{remark:uniform_boundedness}, that the distributions $T^s$
	have a limit in $\mathscr D' ( \mathbf R^n, Y )$ as $s \to 0+$, whose
	first order partial derivatives vanish, and thus is induced, via the
	$\mathscr L^n$ integral, by some $\psi \in Y^\ast$ by
	\cite[4.1.4]{MR41:1976}; the postscript then follows using
	\ref{remark:distributional_derivative}.
\end{remark}

\begin{remark} \label{remark:cz_diff_diff}
	The preceding theorem is partly analogous to
	\cite[Theorem~11]{MR0136849}, where differentiability in Lebesgue
	spaces with respect to $\mathscr L^n$ is treated for $k=1$ in such a
	manner as to include embedding results.
\end{remark}

\begin{remark}
	By
	\ref{lemma:distribution_function}\,\eqref{item:distribution_function:norm_diff}
	and~\ref{example:distribution_representable_by_integration}, taking
	$i=0=l$ in~\eqref{item:pt_diff_derivatives:integral} yields the
	differentiability result for real valued functions on~$\mathbf R^n$
	whose $k$-th order distributional partial derivatives are
	representable by integration contained in \cite[Theorem
	1]{MR0225159-english}.
\end{remark}

\begin{corollary} \label{corollary:characterisation_pt_diff}
	$T$ is pointwise differentiable of order~$k$ at~$a$
	if and only if, for $o \in \boldsymbol \Xi (n,k)$, the distribution
	$\Der^o T$ is pointwise differentiable of order~$0$ at~$a$.
\end{corollary}

\begin{proof}
	Combine \ref{remark:distributional_derivative}
	and~\ref{thm:pt_diff_derivatives}\,\eqref{item:pt_diff_derivatives:top_integral}.
\end{proof}

\begin{remark} \label{remark:point_to_point}
	For the case $k=1=n$, $Y = \mathbf C$, and $\mathbf C$~linear $T$, a
	related characterisation was noted in~\cite[Proposition 2]{MR3261225}.
	The weaker statement that pointwise differentiability of order~$0$ at
	$a$ of the distributions $\Der^o T$ for $o \in \boldsymbol \Xi (n,k)$
	implies pointwise differentiability of $T$ of order~$0$ at~$a$ had
	been proven for the same case in~\cite[3.2 Théorème]{MR0087905} and
	follows for general~$n$ from \cite[Lemma~3]{MR0209835}.
\end{remark}

\section{Differentiability theory}

In this section, we firstly carry out the necessary adaptations of various
known results to cover the case of infinite dimensional target spaces
in~\ref{lemma:weakly_diff_cka}--\ref{remark:fed_3.1.15_revisited}.  Then, we
establish the main theorems of the differentiability theory
in~\ref{thm:cka}--\ref{thm:k_lusin_approximation}.

\begin{lemma} \label{lemma:weakly_diff_cka}
	Suppose $k$ is a positive integer, $Y$ is a separable Banach space, $f
	: \mathbf R^n \to Y^\ast$ is of class~$k-1$, $M = \Lip \Der^{k-1}
	f < \infty$, and $Z_k$ is defined and endowed with the $Y \otimes
	\bigodot_k \mathbf R^n$~topology as in~\ref{miniremark:setup}.

	Then, there exists an $\mathscr L^n$~almost unique, $Z_k$~valued
	$\mathscr L^n$~measurable function~$g$ satisfying $( \mathscr
	L^n)_{(\infty)} ( \| g \| ) \leq M$ and
	\begin{equation*}
		{\textstyle\int} \langle \eta, \Der^k \zeta (x)
		\rangle \langle y, f(x) \rangle \ud \mathscr L^n \, x = (-1)^k
		{\textstyle\int} \zeta (x) \big \langle y, \langle \eta,
		g(x) \rangle \big \rangle \ud \mathscr L^n \, x
	\end{equation*}
	whenever $\zeta \in \mathscr D ( \mathbf R^n, \mathbf R )$, $\eta \in
	\bigodot_k \mathbf R^n$, and $y \in Y$.
\end{lemma}

\begin{proof}
	If $f$ is of class~$k$, then we may take $g = \Der^k f$ and, by
	\cite[2.2.7, 3.1.1, 3.1.11]{MR41:1976}, notice that $M = \sup \im \|
	\Der^k f \|$.  In the general case, the image of~$f$ is contained in a
	separable subspace of~$Y^\ast$, hence convolution as described
	in~\cite[4.1.2]{MR41:1976} yields a sequence of functions $f_j :
	\mathbf R^n \to Y^\ast$ of class~$k$ with $\sup \im \| \Der^k f_j \|
	\leq M$ converging locally uniformly to~$f$ as $j \to \infty$.  Then,
	by~\cite[2.1]{MR3626845} and \ref{miniremark:lebesgue_duality} both
	with $\mu$ and $Y$ replaced by $\mathscr L^n$ and $Y \otimes
	\bigodot_k \mathbf R^n$, and \ref{miniremark:setup} with $Y$ and $m$
	replaced by $L_1 ( \mathscr L^n, Y \otimes \bigodot_k \mathbf R^n )$
	and~$0$, we infer, possibly passing to a subsequence, the existence of
	a $Z_k$~valued $\mathscr L^n$~measurable function~$g$ satisfying $(
	\mathscr L^n)_{(\infty)} ( \| g \| ) \leq M$ and
	\begin{equation*}
		\lim_{j \to \infty} {\textstyle\int} \langle \theta, \Der^k
		f_j \rangle \ud \mathscr L^n = {\textstyle\int} \langle
		\theta, g \rangle \ud \mathscr L^n \quad {\textstyle\text{for
		$\theta \in \mathbf L_1 \big ( \mathscr L^n, Y \otimes
		\bigodot_k \mathbf R^n \big )$}}.
	\end{equation*}
	In view of the special case, the conclusion now readily follows.
\end{proof}

\begin{lemma} \label{lemma:rademacher_cka}
	Suppose $k$ is a positive integer, $Y$ is a separable Banach space, $f
	: \mathbf R^n \to Y^\ast$ is of class~$(k-1,1)$, and $S \in \mathscr
	D' (\mathbf R^n, Y )$ satisfies
	\begin{equation*}
		S ( \phi ) = {\textstyle \int \langle \phi, f \rangle \ud
		\mathscr L^n} \quad \text{for $\phi \in \mathscr D ( \mathbf
		R^n, Y )$}.
	\end{equation*}

	Then, the following three statements hold.
	\begin{enumerate}
		\item \label{item:rademacher_cka:taylor} For $a \in \mathbf
		R^n$, $S$ is $\boldsymbol \nu_{\mathbf B(0,1)}^0$~pointwise
		differentiable of order $(k-1,1)$ at~$a$ and $\pt \Der^m S(a)
		= \pt \Der^m f(a)$ for $m = 0, \ldots, k-1$.
		\item \label{item:rademacher_cka:diff} For $\mathscr L^n$
		almost all $a$, $S$ is pointwise differentiable of order~$k$
		at~$a$.
		\item \label{item:rademacher_cka:dual_sep} If $Y^\ast$ is
		separable, then $S$ is $\boldsymbol \nu_{\mathbf B (0,1)}^0$
		pointwise differentiable of order~$k$ at $\mathscr L^n$~almost
		all~$a$.
	\end{enumerate}
\end{lemma}

\begin{proof}
	In view of~\ref{lemma:weakly_diff_cka}, one may apply
	\ref{lemma:distribution_function}\,\eqref{item:distribution_function:weak_lebesgue}\,\eqref{item:distribution_function:norm}\,\eqref{item:distribution_function:norm_diff}
	with $S$ replaced by $\Der^\xi T$ whenever $\xi \in
	\boldsymbol \Xi (n,k)$.  Therefore, the conclusion follows from
	\ref{thm:pt_diff_derivatives}\,\eqref{item:pt_diff_derivatives:integral}\,\eqref{item:pt_diff_derivatives:hoelder}\,\eqref{item:pt_diff_derivatives:top_integral}.
\end{proof}

\begin{theorem} [Gelfand's Rademacher theorem] \label{thm:rademacher_gelfand}
	Suppose $Y$ is a Banach space, $Y^\ast$ is separable, and $f : \mathbf
	R^n \to Y^\ast$ is locally Lipschitzian.

	Then, $f$ is differentiable at $\mathscr L^n$ almost all $a$.
\end{theorem}

\begin{proof}
	By \cite[5.12\,(i)]{MR1727673}, $Y^\ast$ has the Radon-Nikodým
	property in the sense of~\cite[5.4]{MR1727673}.  Therefore, the
	conclusion follows from \cite[4.3, 6.41]{MR1727673}.
\end{proof}

\begin{remark}
	The main case $n=1$ originates from~\cite[\S\,9, Satz~2,
	Hilfssatz~5]{msb_v46_i2_p235}.  By~\ref{miniremark:setup}, that case
	also follows from the fact $f(b)-f(a) = \int_a^b \langle 1, g
	(x) \rangle \ud \mathscr L^1 x$ for $- \infty < a < b < \infty$, where
	$g$ is as in~\ref{lemma:weakly_diff_cka} with $k=1$, via \cite[2.8.18,
	2.9.9]{MR41:1976}.
\end{remark}

\begin{theorem}	[Whitney's extension theorem] \label{thm:whitney_cka}
	Suppose $Y$ is a normed vectorspace, $k$ is a nonnegative integer, $0
	< \alpha \leq 1$, $A$ is a closed subset of~$\mathbf R^n$, and to
	each $a \in A$ corresponds a polynomial function
	\begin{equation*}
		\text{$P_a : \mathbf R^n \to Y$ with $\degree P_a \leq k$}.
	\end{equation*}
	Whenever $C \subset A$ and $\delta > 0$, let $\varrho (C,\delta)$ be
	the supremum of the set of all numbers
	\begin{gather*}
		\| \Der^m P_a(b) - \Der^m P_b(b) \| \cdot |a-b|^{m-k} \cdot
		(k-m)!
	\end{gather*}
	corresponding to $m = 0, \ldots, k$ and $a,b \in C$ with $0 < |a-b|
	\leq \delta$.

	If $\limsup_{\delta \to 0+} \delta^{-\alpha} \varrho (C,\delta) <
	\infty$ for each compact subset~$C$ of~$A$, then there exists a map $g
	: \mathbf R^n \to Y$ of class~$(k,\alpha)$ such that
	\begin{equation*}
		\Der^m g(a) = \Der^m P_a(a) \quad \text{for $m = 0, \ldots, k$
		and $a \in A$}.
	\end{equation*}
\end{theorem}

\begin{proof}
	Assuming $A \neq \varnothing$, we proceed as in
	\cite[3.1.14]{MR41:1976}; hence, $U = \mathbf R^n \without A$ and
	there is a map $g : \mathbf R^n \to Y$ of class~$k$ such that $g | U$
	is of class~$\infty$ and
	\begin{equation*}
		\Der^m g(a) = \Der^m P_a(a) \quad \text{for $m = 0, \ldots, k$
		and $a \in A$}.
	\end{equation*}
	Noting that lines 19--23 of \cite[p.~226]{MR41:1976} remain valid if
	$i = k+1$, we record, from the proof of \cite[3.1.14]{MR41:1976}, the
	two estimates
	\begin{gather*}
		\| \Der^{k+1} g (x) \| \leq M_{k+1} \dist (x,A)^{-1} \varrho (
		C, 6 \dist (x,A) ) \quad \text{if $x \in U$}, \\
		\| \Der^k g (x) - \Der^k g(a) \| \leq (M_k+1) \varrho
		(C,6|x-a|)
	\end{gather*}
	whenever $a \in A$, $x \in \mathbf B (a,1/3)$, and $C = A \cap \mathbf
	B (a,2)$, where $M_k$ and $M_{k+1}$ are real numbers determined by $n$
	and $k$.

	Suppose $a \in A$, $C = A \cap \mathbf B (a,3)$, notice that there
	exists $0 \leq \kappa < \infty$ with $\varrho (C,\delta) \leq \kappa
	\delta^\alpha$ for $0 < \delta \leq 2$, and let $y,z \in \mathbf B
	(a,1/9)$.  If
	\begin{equation*}
		\dist (y+t(z-y),A) \geq |y-z|/2 \quad \text{for $0 \leq t \leq
		1$},
	\end{equation*}
	we obtain from the first estimate that
	\begin{equation*}
		\| \Der^k g (y)- \Der^k g (z) \| \leq 12 \kappa M_{k+1}
		|y-z|^\alpha
	\end{equation*}
	using~\cite[2.2.7, 3.1.1]{MR41:1976} and $\| \Der \Der^k g \| = \|
	\Der^{k+1} g \|$.  If
	\begin{equation*}
		\dist (y+t(z-y),A) < |y-z|/2 \quad \text{for some $0 \leq t
		\leq 1$},
	\end{equation*}
	then, we take $c \in A$ with $|y+t(z-y)-c| < |y-z|/2$, notice
	\begin{equation*}
		\sup \{ |y-c|, |z-c| \} \leq 3 |y-z|/2 \leq 1/3, \quad
		|c-a| \leq |c-y| + |y-a| \leq 1,
	\end{equation*}
	and obtain from the second estimate (with $a$ replaced by $c$)
	\begin{align*}
		\| \Der^k g (z) - \Der^k g (y) \| & \leq \| \Der^k g(z) -
		\Der^k g(c) \| + \| \Der^k g(c) - \Der^k g(y) \| \\
		& \leq 6 \kappa (M_k+1) ( |z-c|^\alpha + |y-c|^\alpha ) \leq
		18 \kappa (M_k+1) |z-y|^\alpha.
	\end{align*}
	Accordingly, the map $(\Der^k g )| \mathbf B ( a,1/9 )$ is Hölder
	continuous with exponent~$\alpha$.
\end{proof}

\begin{remark}
	The last paragraph of the proof is adapted from
	\cite[VI.2.2.2]{MR0290095}, where $k = 0$ and $Y = \mathbf R$; see
	also \cite[VI.2.3.1--VI.2.3.3]{MR0290095} for general $k$ and $Y =
	\mathbf R$.  In contrast to that source or, for the case $(k,\alpha) =
	(0,1)$ and $Y$ a Banach space, to \cite[Theorem 2]{MR852474}, our
	result concerns the extension to a function satisfying a local instead
	of a global regularity property.
\end{remark}

\begin{theorem} \label{thm:fed_3.1.15_revisited}
	Suppose $k$ is a nonnegative integer, $Z$ is a separable Banach space,
	$f : \mathbf R^n \to Z$ is of class~$k$, and $A = \dmn \Der \Der^k f$.

	Then, for each $\epsilon > 0$, there exists a map $g : \mathbf R^n
	\to Z$ of class~$k+1$ such that
	\begin{equation*}
		\mathscr L^n \big ( A \without \{ a \with \textup{$\Der^m
		f(a)= \Der^m g(a)$ for $m = 0, \ldots, k$} \} \big ) = 0.
	\end{equation*}
\end{theorem}

\begin{proof}
	We may proceed as in~\cite[3.1.15]{MR41:1976}; in fact, it is
	sufficient to replace ``Assured \ldots\ measurable'', $m$, and
	$\mathbf R^n$ in its proof by ``Assured by \cite[3.1.1]{MR41:1976}
	that $A$ is a Borel set and $\Der \Der^k f$ is a Borel function'',
	$n$, and $Z$, respectively.
\end{proof}

\begin{remark} \label{remark:fed_3.1.15_revisited}
	In contrast to~\cite[3.1.15]{MR41:1976}, one may not replace $A$ by
	\begin{equation*}
		\mathbf R^n \cap \{ a \with \limsup_{x\to a} |f(x)-f(a)|/|x-a|
		< \infty \};
	\end{equation*}
	in fact, in view of~\cite[2.1]{MR3626845}, we obtain a separable
	Banach space $Z$ and a Lipschitzian function $f : \mathbf R \to Z$
	with $\mathscr L^1 ( \mathbf R \without \dmn \Der f ) > 0$
	from~\cite[p.\,265]{msb_v46_i2_p235}, observe $\dmn \Der f = \dmn \ap
	\Der f$ by~\cite[3.1.5]{MR41:1976}, and note that, for $g : \mathbf R
	\to Z$ of class~$1$, \cite[2.8.18, 2.9.11]{MR41:1976} yields $\mathscr
	L^1 ( \{ a \with f(a)=g(a) \} \without \ap \Der f ) = 0$.
\end{remark}

\begin{theorem} \label{thm:cka}
	Suppose $k$ is an integer, $0 < \alpha \leq 1$, $k + \alpha \geq 0$,
	$Y$ is a Banach space, $T \in \mathscr D' ( \mathbf R^n, Y )$, and
	$A$~is the set of points at which $T$ is pointwise differentiable of
	order~$(k,\alpha)$.

	Then, there exists a sequence of compact subsets $C_j$ of~$\mathbf
	R^n$ with $A = \bigcup_{j=1}^\infty C_j$ and, if $k \geq 0$, also a
	sequence of functions $f_j : \mathbf R^n \to Y^\ast$ of
	class~$(k,\alpha)$ satisfying
	\begin{equation*}
		\pt \Der^m T (a) = \Der^m f_j (a) \quad \text{for $a \in C_j$
		and $m = 0, \ldots, k$}
	\end{equation*}
	whenever $j$ is a positive integer.
\end{theorem}

\begin{proof}
	Suppose $Z_m$ are defined and endowed with the $Y \otimes \bigodot_m
	\mathbf R^n$~topology as in~\ref{miniremark:setup}.  Whenever $j$ is a
	positive integer, we let $\nu_j = \boldsymbol \nu_{\mathbf
	B(0,1)}^j | \mathscr D_{\mathbf B(0,1)} ( \mathbf R^n, Y )$ and define
	compact sets~$L_j$ to consist of those $(a,\psi) \in \mathbf R^n
	\times \bigoplus_{m=0}^k Z_m$ satisfying
	\begin{gather*}
		|a| \leq j, \qquad \| \psi_m \| \leq j \quad \text{for $m = 0,
		\ldots, k$}, \\
		\left | T_x \big ( \phi ( {\textstyle\frac{x-a}r}
		) \big ) - \sum_{m=0}^k \left \langle {\textstyle \int
		\phi \left ( {\textstyle\frac{x-a}r} \right ) \otimes
		\frac{(x-a)^m}{m!} \ud \mathscr L^n \, x} , \psi_m \right
		\rangle \right | \leq j r^{k+\alpha+n}
	\end{gather*}
	whenever $0 < r \leq j^{-1}$, $\phi \in \mathscr D_{\mathbf B (0,1)} (
	\mathbf R^n, Y )$, and $\nu_j ( \phi ) \leq 1$.%
	\begin{footnote}
		{Concerning the case $k=-1$, recall that the direct sum of an
		empty family of vectorspaces is the zero vectorspace and that
		the sum over the empty set equals zero.}
	\end{footnote}%
	Then, we have
	\begin{equation*}
		\left ( A\times \bigoplus_{m=0}^k Z_m \right ) \cap \{ (a,\psi)
		\with \text{$\psi_m = \pt \Der^m T(a)$ for $m = 0, \ldots, k$}
		\} = \bigcup_{j=1}^\infty L_j
	\end{equation*}
	by~\ref{remark:norm_differentiability}.  We define compact sets $C_j =
	\{ a \with \text{$(a,\psi) \in L_j$ for some $\psi$} \}$, whenever $j$
	is a positive integer, as well as polynomial functions $P_a : \mathbf
	R^n \to Y^\ast$ by letting
	\begin{equation*}
		P_a(x) = \sum_{m=0}^k \langle (x-a)^m/m!, \pt \Der^m T(a)
		\rangle \quad \text{for $a \in A$ and $x \in \mathbf R^n$}.
	\end{equation*}

	To apply~\ref{thm:whitney_cka}, we suppose $j$ is a positive integer
	and $a, b \in C_j$, $0 < |a-b| \leq 1/j$.  Abbreviating $r = |b-a|$,
	$v = r^{-1}(b-a)$, and $K = \mathbf B (v,1) \cap \mathbf B (0,1)$, we
	notice $ r^{-1} (x-a) = r^{-1} (x-b) + v$ for $x \in \mathbf R^n$ and
	infer
	\begin{align*}
		\big | {\textstyle\int \langle \phi ( \frac{x-a}r),
		(P_a-P_b)(x) \rangle \ud \mathscr L^n \, x}  \big | & \leq
		\big | {\textstyle\int \langle \phi ( \frac{x-a}r ), P_a(x)
		\rangle \ud \mathscr L^m \, x - T_x \big ( \phi ( \frac{x-a}r
		) \big)} \big | \\
		& \phantom \leq \ + \big | {\textstyle T_x \big ( \theta (
		\frac{x-b}r ) \big ) - \int \langle \theta ( \frac{x-b}r
		), P_b (x) \rangle \ud \mathscr L^m \, x} \big | \\
		& \leq 2 j r^{k+\alpha+n}
	\end{align*}
	whenever $\phi \in \mathscr D_K ( \mathbf R^n, Y )$ and $\nu_j ( \phi)
	\leq 1$, where $\theta \in \mathscr D_{\mathbf B(0,1)} (\mathbf R^n, Y
	)$ satisfies $\theta (x) = \phi(x+v)$ for $x \in \mathbf R^n$.
	Consequently, applying~\ref{lemma:polynomial_fcts} with $\nu$ and
	$P(x)$ replaced by $\nu_j$ and $\langle y, (P_a-P_b)(a+rx)\rangle$,
	whenever $y \in Y$, we conclude
	\begin{equation*}
		\sup \{ r^m \| \Der^m (P_a-P_b)(x) \| \with x \in \mathbf B
		(a,r), m = 0, \ldots, k \} \leq 2j \kappa r^{k+\alpha},
	\end{equation*}
	where $0 \leq \kappa < \infty$ is determined by $j$, $k$, and~$n$;
	hence, \ref{thm:whitney_cka} yields the conclusion.
\end{proof}

\begin{corollary}
	The set $A$ is a Borel subset of~$\mathbf R^n$ and the functions
	mapping $a \in A$ onto $\Der^m T (a) \in \bigodot^m ( \mathbf R^n,
	Y^\ast )$ are Borel functions for $m = 0, \ldots, k$.
\end{corollary}

\begin{remark} \label{remark:zielezny_sigma_compact}
	In case $n=1$, $k=-1$, $\alpha=1$, $Y = \mathbf C$, and $\mathbf
	C$~linear $T$, the fact that $A$ is the union of a countable family of
	compact sets was already obtained in~\cite[\S\,3.1
	Folgerung~2]{MR0113134}.
\end{remark}

\begin{theorem} \label{thm:borel_derivatives}
	Suppose $k$ is a nonnegative integer, $Y$ is a separable Banach space,
	$Z_k$ is defined and endowed with the $Y \otimes \bigodot_k \mathbf
	R^n$ topology as in \ref{miniremark:setup}, $T \in \mathscr D' (
	\mathbf R^n, Y )$, and $A$ is the set of points at which $T$~is
	pointwise differentiable of order~$k$.
	
	Then, $A$~is a Borel set and the function mapping $a \in A$ onto $\pt
	\Der^k T (a) \in Z_k$ is a Borel function.
\end{theorem}

\begin{proof}
	Suppose $0 \leq M < \infty$.  It is sufficient to prove the assertion
	with $A$ replaced by
	\begin{equation*}
		A' = A \cap \{ a \with \text{$\| \pt \Der^m T(a) \| \leq M$
		for $m = 0, \ldots, k$} \}.
	\end{equation*}
	For this purpose, let $B$ be the Borel set (see~\ref{thm:cka}) of
	points at which $T$~is pointwise differentiable of order $(k-1,1)$.
	Moreover, whenever $i$ and $j$ are positive integers and $\phi \in
	\mathscr D ( \mathbf R^n, Y )$, we define the closed sets
	$L(i,j,\phi)$ to consist of those $(a,\psi) \in \mathbf R^n \times
	\bigoplus_{m=0}^k Z_m$, satisfying
	\begin{gather*}
		\| \psi_m \| \leq M \quad \text{for $m = 0, \ldots, k$}, \\
		\left | T_x \big ( \phi ( {\textstyle\frac{x-a}r} ) \big ) -
		{\textstyle\int} \big \langle \phi ( {\textstyle\frac{x-a}r}
		), \sum_{m=0}^k \langle (x-a)^m/m!, \psi_m \rangle \big
		\rangle \ud \mathscr L^n \, x \right | \leq r^{k+n}/i
	\end{gather*}
	for $0 < r \leq 1/j$, where $Z_m$ are defined and endowed with the $Y
	\otimes \bigodot_m \mathbf R^n$~topology as in~\ref{miniremark:setup}.
	Employing \cite[2.2, 2.24]{MR3528825} to choose a countable
	sequentially dense subset~$\Delta$ of $\mathscr D ( \mathbf R^n, Y )$,
	we infer that
	\begin{equation*}
		G = \left ( B \times \bigoplus_{m=0}^k Z_m \right ) \cap
		\bigcap_{\phi \in \Delta} \bigcap_{i=1}^\infty
		\bigcup_{j=1}^\infty L(i,j,\phi)
	\end{equation*}
	is a Borel set.  Noting $(a,\psi) \in G$ if and only if $a \in A'$ and
	$\psi_m = \pt \Der^m T (a)$ for $m = 0, \ldots, k$
	by~\ref{lemma:dense_diff_crit}, we apply
	\cite[4.1]{MR3936235} with $X$ and $Y$ replaced by $\mathbf
	R^n$ and $\big ( \bigoplus_{m=0}^k Z_m \big ) \cap \{ \psi \with
	\text{$\| \psi_m \| \leq M$ for $m = 0, \ldots, k$} \}$ to infer the
	conclusion.
\end{proof}

\begin{remark}
	In view of \ref{remark:value_of_some_order}
	and~\ref{remark:dense_diff_crit}, in case $n=1$, $k=0$, $Y = \mathbf
	C$, and $\mathbf C$~linear $T$, slightly more precise information
	on~$A$ may be obtained from~\cite[\S\,3.3 Folgerung~1]{MR0113134}.
\end{remark}

\begin{remark} \label{remark:gelfands_classic}
	From \ref{example:nonmeasurable}
	and~\ref{lemma:distribution_function}\,\eqref{item:distribution_function:weak_lebesgue}
	we infer that, with respect to the norm topology on~$Z_k$, the
	function $\pt \Der^k T$ may be $\mathscr L^n \restrict
	A$~nonmeasurable.
\end{remark}

\begin{lemma} \label{lemma:big_O_little_o}
	Suppose $n$ is a positive integer, $A$ is a closed subset of~$\mathbf
	R^n$, $Y$~is a Banach space, $T \in \mathscr D' ( \mathbf R^n, Y )$,
	$A \cap \spt T = \varnothing$, $0 < r < \infty$, $0 \leq \kappa <
	\infty$, $0 \leq \lambda < \infty$, $i$ is a nonnegative integer, and
	\begin{equation*}
		| T ( \phi ) | \leq \kappa s^{n+\lambda+i} \sup \im \| \Der^i
		\phi \|
	\end{equation*}
	whenever $b \in A$, $0 < s \leq 3r$, and $\phi \in \mathscr
	D_{\mathbf B (b,s)} ( \mathbf R^n, Y )$.

	Then, for some $0 \leq \Gamma < \infty$ determined by $n$, $\lambda$,
	and $i$, there holds
	\begin{equation*}
		| T ( \phi ) | \leq \Gamma \kappa r^{\lambda+i} \mathscr L^n
		( \mathbf B ( a,3r ) \without A ) \sup \im \| \Der^i \phi \|
	\end{equation*}
	whenever $a \in A$ and $\phi \in \mathscr D_{\mathbf B (a,r)} (
	\mathbf R^n, Y )$.
\end{lemma}

\begin{proof}
	We assume $r<1/2$ and suppose $a \in A$ and $\phi \in \mathscr
	D_{\mathbf B(a,r)} ( \mathbf R^n, Y )$.
	
	We define $h : \mathbf R^n \to \mathbf R$ by
	\begin{equation*}
		h(x) = {\textstyle\frac 1{20}} \inf \{ 1, \dist (x,A) \} \quad
		\text{for $x \in \mathbf R^n \without A$}.
	\end{equation*}
	Applying~\cite[3.1.13]{MR41:1976} with $m$ and $\Phi$ replaced by $n$
	and $\{ \mathbf R^n \without A \}$, we obtain a countable subset $C$
	of $\mathbf R^n \without A$, nonnegative functions $\zeta_c \in
	\mathscr D_{\mathbf B(c,10h(c))} ( \mathbf R^n, \mathbf R )$ for $c
	\in C$, and a sequence, determined by $n$, of numbers $0 \leq V_m <
	\infty$ for every nonnegative integer $m$ such that the family $\{
	\mathbf B (c,h(c)) \with c \in C \}$ is disjointed, such that
	\begin{gather*}
		h(x) \geq
		h(c)/3 \quad \text{for $x \in \mathbf B(c,10h(c))$}, \\
		\card ( C \cap \{ c \with \mathbf B (c,10h(c)) \cap \mathbf B
		(x,10h(x)) \neq \varnothing \} ) < \infty \quad \text{for $x
		\in \mathbf R^n \without A$}, \\
		\| \Der^m \zeta_c(x) \| \leq V_m h(x)^{-m} \quad \text{for $x
		\in \mathbf R^n \without A$, $m = 0,1,2, \ldots$}
	\end{gather*}
	for $c \in C$, and such that $\sum_{c \in C} \zeta_c (x) = 1$ for $x
	\in \mathbf R^n \without A$.  We let
	\begin{equation*}
		C' = C \cap \{ c \with \mathbf B (c,10h(c)) \cap \mathbf
		B(a,r) \cap \spt T \neq \varnothing \}.
	\end{equation*}
	Since $\mathbf B(a,r) \cap \spt T$ is a compact subset of $\mathbf R^n
	\without A$, the set $C'$ is finite and $\mathbf B (a,r) \cap \spt T$
	is contained in the interior of $\{x \with \sum_{c \in C'} \zeta_c (x)
	= 1 \}$, whence it follows
	\begin{equation*}
		T ( \phi ) = \sum_{c \in C'} T ( \zeta_c \phi ).
	\end{equation*}
	Next, we choose $\xi : C' \to A$ with $|c-\xi(c)| = \dist (c,A)$ and
	verify
	\begin{equation*}
		10h(c) \leq r, \quad \mathbf B(c,10h(c)) \subset \mathbf B (
		\xi(c),30h(c)), \quad \mathbf B (c,h(c)) \subset \mathbf B
		(a,3r) \without A
	\end{equation*}
	for $c \in C'$; in fact, we have
	\begin{equation*}
		20h(c) \leq |c-\xi(c)| \leq |c-a| \leq 10h(c)+r \leq 2r < 1,
		\quad 20h(c) = |c-\xi(c)|.
	\end{equation*}

	Whenever $c \in C$, considering both $\zeta_c$ and $\theta$ as maps
	into the normed algebra~$\bigodot_\ast Y$, the general Leibniz formula
	in~\cite[3.1.11]{MR41:1976} yields
	\begin{equation*}
		\Der^i ( \zeta_c \phi ) (x) = \sum_{m=0}^i \Der^m \zeta_c (x)
		\odot \Der^{i-m} \phi (x) \quad \text{for $x \in \mathbf
		R^n$}.
	\end{equation*}
	Using~\cite[1.10.5]{MR41:1976} and~\ref{miniremark:lip}, we therefore
	estimate
	\begin{align*}
		\sup \im \| \Der^i (\zeta_c \theta) \| & \leq \sum_{m=0}^i
		\binom im V_m 3^m h(c)^{-m} r^m \sup \im \| \Der^i \phi \| \\
		& \leq \Delta h(c)^{-i} r^i\sup \im \| \Der^i \phi \|
	\end{align*}
	whenever $c \in C'$, where $\Delta = \sum_{m=0}^i \binom im V_m 3^m$.
	Accordingly, abbreviating $\Gamma = \Delta (30)^{n+i} 3^\lambda
	\boldsymbol \alpha (n)^{-1}$, we obtain
	\begin{align*}
		| T ( \zeta_c \phi ) | & \leq \kappa (30h(c))^{n+\lambda+i}
		\sup \im \| \Der^i ( \zeta_c \phi ) \| \\
		& \leq \Gamma \kappa r^{\lambda+i} \mathscr L^n \, \mathbf B
		(c,h(c)) \sup \im \| \Der^i \phi \|
	\end{align*}
	for $c \in C'$, whence $| T ( \phi ) | \leq \Gamma \kappa
	r^{\lambda+i} \mathscr L^n ( \mathbf B(a,3r) \without A ) \sup \im \|
	\Der^i \phi \|$ follows.
\end{proof}

\begin{remark}
	The preceding lemma extends~\cite[A.1]{MR3023856} where the case
	$\lambda = 0$, $i=1$, and $Y$ a finite dimensional Euclidean space is
	treated.
\end{remark}

\begin{theorem} \label{thm:big_O_little_o}
	Suppose $i$ and $k$ are nonnegative integers, $Y$ is a separable
	Banach space, $R \in \mathscr D' ( \mathbf R^n, Y )$, and $A$ is the
	set of all $a \in \mathbf R^n$ such that $R$ is $\boldsymbol
	\nu_{\mathbf B (0,1)}^i$~pointwise differentiable of order~$(k-1,1)$
	at $a$ and $\pt \Der^m R(a) = 0$ for $m = 0, \ldots, k-1$.

	Then, $\mathscr L^n$~almost all $a \in A$ satisfy the following three
	statements:
	\begin{enumerate}
		\item \label{item:big_O_little_o:minimum} The distribution~$R$
		is pointwise differentiable of order~$k$ at~$a$.
		\item \label{item:big_O_little_o:dual_sep} If $k>0$ or
		$Y^\ast$ is separable, then $R$ is $\boldsymbol \nu_{\mathbf
		B(0,1)}^i$ pointwise differentiable of order~$k$ at~$a$.
		\item \label{item:big_O_little_o:diff} If $k>0$, then $\pt
		\Der^k R(a)=0$.
	\end{enumerate}
\end{theorem}

\begin{proof}
	Whenever $j$ is a positive integer, we define $A_j$ to be the set of
	all $b \in \mathbf R^n$ satisfying
	\begin{equation*}
		| R ( \phi ) | \leq j s^{n+k+i} \sup \im \| \Der^i \phi \|
		\quad \text{for $0 < s \leq 6/j$ and $\phi \in \mathscr
		D_{\mathbf B (b,s)} ( \mathbf R^n, Y )$}.
	\end{equation*}
	By~\ref{miniremark:lip}, a point $b$ belongs to $A_j$ if and only if
	\begin{equation*}
		\big | R_x ( \phi ( {\textstyle\frac{x-b}s} ) ) \big | \leq j
		s^{n+k} \boldsymbol \nu_{\mathbf B(0,1)}^i ( \phi ) \quad
		\text{for $0 < s \leq 6/j$ and $\phi \in \mathscr D_{\mathbf
		B(0,1)} ( \mathbf R^n, Y )$};
	\end{equation*}
	in particular, the sets $A_j$ are closed and $A = \bigcup_{j=1}^\infty
	A_j$.

	Supposing $j$ to be a positive integer, the conclusion will be shown
	to hold at $\mathscr L^n$~almost all $a \in A_j$.  We choose $\Phi \in
	\mathscr D ( \mathbf R^n, \mathbf R)$ satisfying $\int \Phi \ud
	\mathscr L^n = 1$ and $\spt \Phi \subset \mathbf B (0,1)$ and
	abbreviate $\Delta = \sup \im \| \Der^i \Phi \|$.

	In this paragraph, we study various objects whenever $0 < \epsilon
	\leq 3/j$.  With $\Phi_\epsilon (x) = \epsilon^{-n} \Phi (
	\epsilon^{-1} x )$ for $x \in \mathbf R^n$, we define
	$R_\epsilon \in \mathscr D' ( \mathbf R^n, Y )$ by
	\begin{equation*}
		R_\epsilon ( \phi ) = R ( \Phi_\epsilon \ast \phi )
		\quad \text{for $\phi \in \mathscr D ( \mathbf R^n, Y )$}.
	\end{equation*}
	From~\cite[4.1.2]{MR41:1976}, we infer the existence of $f_\epsilon
	\in \mathscr E ( \mathbf R^n, Y^\ast )$ satisfying
	\begin{gather*}
		R_\epsilon ( \phi ) = {\textstyle\int} \langle \phi,
		f_\epsilon \rangle \ud \mathscr L^n \quad \text{for $\phi
		\in \mathscr D ( \mathbf R^n, Y )$}, \\
		\langle y, f_\epsilon (x) \rangle = (R)_b
		(
		\Phi_\epsilon (b-x) y ) \quad \text{whenever $x \in
		\mathbf R^n$ and $y \in Y$}.
	\end{gather*}
	Clearly, we have $R_\epsilon \to R$ as $\epsilon \to 0+$ and
	\begin{equation*}
		\| f_\epsilon (x) \| \leq j 2^{n+k+i} \Delta \epsilon^k
		\quad \text{whenever $\dist (x,A_j) \leq \epsilon$}.
	\end{equation*}
	We define $S_\epsilon \in \mathscr D' ( \mathbf R^n, Y )$ by
	\begin{equation*}
		S_\epsilon ( \phi ) = {\textstyle\int_{\{ x \with \dist
		(x,A_j) \leq \epsilon \}}} \langle \phi, f_\epsilon
		\rangle \ud \mathscr L^n \quad \text{for $\phi \in \mathscr D
		( \mathbf R^n, Y )$}
	\end{equation*}
	and abbreviate $T_\epsilon = R_\epsilon - S_\epsilon$.  For
	$b \in A_j$, $0 < s \leq 3/j$, and $\phi \in \mathscr D_{\mathbf B
	(b,s)} ( \mathbf R^n, Y )$ with $\sup \im \| \Der^i \phi \| \leq 1$,
	we estimate
	\begin{gather*}
		\spt ( \Phi_\epsilon \ast \phi ) \subset \mathbf B
		(b,\epsilon+s), \quad \text{$| R_\epsilon ( \phi ) |
		\leq j 2^{n+k+i} s^{n+k+i}$ if $\epsilon \leq s$}, \\
		( \spt T_\epsilon ) \cap \{ x \with \dist (x,A_j) <
		\epsilon \} = \varnothing, \quad \text{$T_\epsilon (
		\phi ) = 0$ if $\epsilon > s$}, \\
		| S_\epsilon ( \phi ) | \leq ( \mathscr L^n \restrict \{ x
		\with \dist (x,A_j) \leq \epsilon \} )_{(\infty)} (
		\mathscr L^n)_{(1)} ( \phi ) \leq j 2^{n+k+i} \Delta
		\epsilon^k \boldsymbol \alpha (n) s^{n+i}, \\
		| T_\epsilon ( \phi ) | \leq \kappa s^{n+k+i}, \quad
		\text{where $\kappa = j 2^{n+k+i} (1+\Delta \boldsymbol \alpha
		(n))$}.
	\end{gather*}
	Therefore, applying~\ref{lemma:big_O_little_o} with $A$, $T$, and
	$\lambda$ replaced by $A_j$, $T_\epsilon$, and $k$, we infer
	\begin{equation*}
		| T_\epsilon ( \phi ) | \leq \Gamma \kappa r^{k+i} \mathscr
		L^n (\mathbf B(a,r) \without A_j) \sup \im \| \Der^i \phi \|
	\end{equation*}
	whenever $a \in A_j$, $0 < r \leq 1/j$, and $\phi \in \mathscr
	D_{\mathbf B (a,r)} ( \mathbf R^n, Y )$.

	Since $( \mathscr L^n \restrict \{ x \with \dist (x,A_j) \leq
	\epsilon \} )_{(\infty)} (f_\epsilon) \leq j 2^{n+k+i} \Delta
	\epsilon^k$ for $0 < \epsilon \leq 3/j$, we may
	use~\cite[2.1]{MR3626845} and \ref{miniremark:lebesgue_duality} both
	with $\mu =\mathscr L^n$ and \ref{miniremark:setup} with $Y$ and $m$
	replaced by $L_1 ( \mathscr L^n, Y )$ and~$0$ to construct $S \in
	\mathscr D' ( \mathbf R^n, Y )$ as a subsequential limit of
	$S_\epsilon$ as $\epsilon \to 0+$ and a $Y^\ast$~valued
	function that is $\mathscr L^n$~measurable with respect to the
	$Y$~topology on~$Y^\ast$ and that satisfies
	\begin{gather*}
		\text{$(\mathscr L^n )_{(\infty)} ( \| f \| ) \leq j 2^{n+i}
		\Delta$ if $k=0$}, \quad \text{$f = 0$ if $k>0$}, \\
		S ( \phi ) = {\textstyle\int} \langle \phi, f \rangle \ud
		\mathscr L^n \quad \text{for $\phi \in \mathscr D ( \mathbf
		R^n, Y )$}.
	\end{gather*}
	Defining $T = R-S$, the final estimate of the preceding paragraph
	implies
	\begin{equation*}
		| T ( \phi ) | \leq \Gamma \kappa r^{k+i} \mathscr L^n
		(\mathbf B(a,r) \without A_j) \sup \im \| \Der^i \phi \|
	\end{equation*}
	for $a \in A_j$, $0 < r \leq 1/j$, and $\phi \in \mathscr D_{\mathbf B
	(a,r)} ( \mathbf R^n, Y )$.  Hence, $T$~is
	$\boldsymbol \nu_{\mathbf B(0,1)}^i$~pointwise differentiable of
	order~$k$ at~$a$ with $\pt \Der^k T(a)=0$ at $\mathscr L^n$~almost all
	$a \in A_j$ by \cite[2.8.18, 2.9.11]{MR41:1976}.  In view
	of~\ref{lemma:distribution_function}\,\eqref{item:distribution_function:weak_lebesgue}\,\eqref{item:distribution_function:lebesgue},
	the conclusion now follows.
\end{proof}

\begin{remark}
	The preceding theorem generalises~\cite[A.3]{MR3023856} where the case
	$i = 1$, $k=0$, and $Y$ a finite dimensional Euclidean space is
	treated.
\end{remark}

\begin{remark} \label{remark:big_O_little_o:sep}
	The hypothesis in~\eqref{item:big_O_little_o:dual_sep} may not be
	omitted by~\ref{remark:need_dual_sep}.
\end{remark}

\begin{remark} \label{remark:big_O_little_o:zero_order}
	Considering the distribution associated to any nonzero continuous
	function $f : \mathbf R \to \mathbf R$, the hypothesis
	in~\eqref{item:big_O_little_o:diff} may not be omitted
	by~\ref{lemma:distribution_function}\,\eqref{item:distribution_function:norm_diff}.
\end{remark}

\begin{corollary} \label{corollary:big_O_little_o}
	Suppose $k$ is a positive integer, $Y$ is a separable Banach space, $R
	\in \mathscr D' ( \mathbf R^n, Y )$, and $A$ is the set of all $a \in
	\mathbf R^n$ such that $R$ is pointwise differentiable of
	order~$(k-1,1)$ at $a$ and $\pt \Der^m R(a) = 0$ for $m = 0, \ldots,
	k-1$.

	Then, there holds $\pt \Der^k R(a) = 0$ for $\mathscr L^n$~almost all
	$a \in A$.
\end{corollary}

\begin{proof}
	We combine \ref{remark:norm_differentiability} with $k$ and $\alpha$
	replaced by $k-1$ and $1$
	and~\ref{thm:big_O_little_o}\,\eqref{item:big_O_little_o:diff}.
\end{proof}

\begin{theorem} \label{thm:rademacher_T}
	Suppose $i$ and $k$ are nonnegative integer, $Y$ is a separable Banach
	space, $T \in \mathscr D' ( \mathbf R^n, Y )$, and $A$ is the set of
	points at which the distribution $T$~is $\boldsymbol \nu_{\mathbf B
	(0,1)}^i$ pointwise differentiable of order~$(k-1,1)$.

	Then, $\mathscr L^n$~almost all $a \in A$ satisfy the following two
	statements:
	\begin{enumerate}
		\item \label{item:rademacher_T:diff} The distribution~$T$ is
		pointwise differentiable of order~$k$ at~$a$.
		\item \label{item:rademacher_T:dual_sep} If $Y^\ast$ is
		separable, then $T$ is $\boldsymbol \nu_{\mathbf
		B(0,1)}^i$~pointwise differentiable of order~$k$ at~$a$.
	\end{enumerate}
\end{theorem}

\begin{proof}
	In view
	of~\ref{thm:big_O_little_o}\,\eqref{item:big_O_little_o:minimum}\,\eqref{item:big_O_little_o:dual_sep},
	we may assume $k \geq 1$.  Then, we obtain sequences $C_j$ and~$f_j$
	from~\ref{thm:cka} with $k$ and $\alpha$ replaced by $k-1$ and $1$.
	Associating $S_j$ with $f_j$ as in~\ref{lemma:rademacher_cka} and
	defining $R_j = T-S_j$, the distributions~$R_j$ are $\boldsymbol
	\nu_{\mathbf B(0,1)}^i$~pointwise differentiable of order~$(k-1,1)$ at
	$a$ with $\pt \Der^m R_j (a) = 0$ whenever $a \in A \cap C_j$, $m = 0,
	\ldots, k-1$, and $j$ is a positive integer by \ref{thm:cka}
	and~\ref{lemma:rademacher_cka}\,\eqref{item:rademacher_cka:taylor}.  The
	conclusion now follows from
	\ref{lemma:rademacher_cka}\,\eqref{item:rademacher_cka:diff}\,\eqref{item:rademacher_cka:dual_sep}
	and~\ref{thm:big_O_little_o}\,\eqref{item:big_O_little_o:minimum}\,\eqref{item:big_O_little_o:dual_sep}.
\end{proof}

\begin{corollary} \label{corollary:rademacher}
	Suppose $k$ is a nonnegative integer, $Y$ is a separable Banach space,
	$T \in \mathscr D' ( \mathbf R^n, Y )$, and $A$ is the set of points
	at which $T$~is pointwise differentiable of order~$(k-1,1)$.

	Then, $T$ is pointwise differentiable of order~$k$ at $\mathscr
	L^n$~almost all $a \in A$.
\end{corollary}

\begin{proof}
	We combine \ref{remark:norm_differentiability}
	and~\ref{thm:rademacher_T}\,\eqref{item:rademacher_T:diff}.
\end{proof}

\begin{remark} \label{remark:zielezny_rademacher}
	The case $n=1$, $k=0$, $Y = \mathbf C$, and $\mathbf C$~linear $T$
	was established in~\cite[\S\,2.1 Satz 2.1]{MR0113134}.
\end{remark}

\begin{theorem} \label{thm:k_lusin_approximation}
	Suppose $k$ is a nonnegative integer, $Y$ is a Banach space,
	$Y^\ast$~is separable, $T \in \mathscr D' ( \mathbf R^n, Y )$, and $A$
	is the set of points at which $T$ is pointwise differentiable of
	order~$k$.

	Then, for each $\epsilon > 0$, there exists $g : \mathbf R^n \to
	Y^\ast$ of class~$k$ such that
	\begin{equation*}
		\mathscr L^n \big ( A \without \{ a \with \textup{$\pt \Der^m
		T (a) = \Der^m g (a)$ for $m = 0, \ldots, k$} \} \big ) <
		\epsilon.
	\end{equation*}
\end{theorem}

\begin{proof}
	By~\cite[II.3.16]{MR0117523}, $Y$ is separable.  In view
	of~\ref{miniremark:setup}, we infer from~\ref{thm:borel_derivatives}
	that $\dmn \Der^m T$ is a Borel set and $\pt \Der^m T$ is a Borel
	function with respect to the norm topology on~$\bigodot^m ( \mathbf
	R^n, Y )$ for $m = 0, \ldots, k$.  If $k = 0$, we may accordingly
	combine \cite[2.3.5, 3.1.14]{MR41:1976} to obtain the conclusion.

	Suppose now $k > 0$.  Then, we employ~\ref{thm:cka} with $k$ and
	$\alpha$ replaced by $k-1$ and $1$ to construct $f : \mathbf R^n \to
	Y^\ast$ of class~$(k-1,1)$ such that
	\begin{equation*}
		\mathscr L^n \big ( A \without \{ a \with \text{$\pt \Der^m T
		(a) = \pt \Der^m f (a)$ for $m = 0, \ldots, k-1$} \} \big ) <
		\epsilon.
	\end{equation*}
	Noting \ref{miniremark:setup} and
	applying~\ref{thm:rademacher_gelfand} with $Y$ and $f$ replaced by $Y
	\otimes \bigodot_{k-1} \mathbf R^n$ and $\Der^{k-1} f$, we infer that
	$\Der^{k-1} f$ is differentiable at $\mathscr L^n$~almost all~$a$.
	Therefore, by~\ref{thm:fed_3.1.15_revisited} with $k$ and $Z$ replaced
	by $k-1$ and~$Y^\ast$, there exists $g : \mathbf R^n \to Y^\ast$ of
	class~$k$ such that
	\begin{equation*}
		\mathscr L^n \big ( A \without \{ a \with \text{$\pt \Der^m T
		(a) = \pt \Der^m g (a)$ for $m = 0, \ldots, k-1$} \} \big ) <
		\epsilon.
	\end{equation*}
	The conclusion now follows from~\ref{corollary:big_O_little_o} applied
	with $R \in \mathscr D' ( \mathbf R^n, Y )$ defined by $R ( \phi ) = T
	( \phi ) - \int \langle \phi, g \rangle \ud \mathscr L^n$ for $\phi
	\in \mathscr D ( \mathbf R^n, Y )$.
\end{proof}

\begin{remark} \label{remark:k_lusin_approximation}
	One cannot replace $Y^\ast$ by $Y$ in the separability hypothesis
	by~\ref{remark:gelfands_classic}.
\end{remark}


\noindent \textsc{Affiliations}

\medskip \noindent Department of Mathematics \\
National Taiwan Normal University \\
No.88, Sec.4, Tingzhou Rd. \\
Wenshan Dist., \textsc{Taipei City 11677 \\
Taiwan(R.\ O.\ C.)}

\medskip \noindent Mathematical Division \\
National Center for Theoretical Sciences \\
No.1, Sec.4, Roosevelt Rd. \\
Da'an Dist., \textsc{Taipei City 10617 \\ Taiwan(R.\ O.\ C.)}

\medskip \noindent \textsc{Email address}

\medskip \noindent
\href{mailto:Ulrich.Menne@math.ntnu.edu.tw}{Ulrich.Menne@math.ntnu.edu.tw}

\end{document}